\documentclass{mcom-l}
\usepackage{amsmath}
\usepackage{amssymb}
\usepackage{graphicx}
\usepackage{tikz}
\usepackage{enumitem}
\usepackage{comment}
\usepackage{bm}
\usepackage{soul}
\newlist{steps}{enumerate}{1}
\setlist[steps, 1]{label = (\roman*)}

\usepackage{color}
\usepackage{soul}
\def\WTD{{D_h}}
\def\WTN{{N_h}}

\def\urho{{\underline p}}
\def\orho{{\overline p}}
\def\ukappa{{\underline q}}

\newcommand{\inn}{\textrm{ in }}
\newcommand{\onn}{\textrm{ on }}

\newcommand{\tu}{\widetilde u}
\newcommand{\nn}{\mathbf n}

\newcommand{\scal}{{\theta}}
\newcommand{\scalh}{{\theta_h}}
\newcommand{\deltah}{{\delta_h}}
\newcommand{\jc}[1]{\textcolor{black}{{#1}}}

\def\trino{|\hspace{-1pt} | \hspace{-1pt}|}
\def\xib{{\bm\xi}}
\def\yb{{\bm y}}
\def\xb{{\bm x}}
\def\etab{{\bm\eta}}
\def\alphab{{\bm\alpha}}

\def\KKK{{\mathcal K}}
\def\DDD{{N}}
\def\EEE{{\mathcal E}}
\def\mmm{{\ell}}
\def\mumu{{\mu}}
\newtheorem{remark}{Remark}

\newtheorem{theorem}{Theorem}
\newtheorem{lemma}{Lemma}
\newtheorem{proposition}{Proposition}

\title[The Polynomial Extension Finite Element Method]{{Optimally accurate higher-order finite element \\methods for polytopial approximations of\\ domains with smooth boundaries}}

\author{James Cheung}
\address{Department of Scientific Computing, Florida State University, Tallahassee FL 32309, USA.} \email{{\tt jc07g@my.fsu.edu}}

\author{Mauro Perego}
\address{Center for Computing Research, Sandia National Laboratories, Albuquerque, NM 87123, USA.}
\email{\tt mperego@sandia.gov}

\author{Pavel Bochev}
\address{Center for Computing Research, Sandia National Laboratories, Albuquerque, NM 87123, USA.}
\email{\tt pbboche@sandia.gov}

\author{Max Gunzburger}
\address{Department of Scientific Computing, Florida State University, Tallahassee FL 32309, USA.}
\email{{\tt mgunzburger@fsu.edu}}

\thanks{{Sandia National Laboratories is a multimission laboratory managed and operated by National Technology and Engineering Solutions of Sandia, LLC., a wholly owned subsidiary of Honeywell International, Inc., for the U.S. Department of Energy's National Nuclear Security Administration under contract DE-NA-0003525.}}
\thanks{This material is based upon work supported by the U.S. Department of Energy, Office of Science, Office of Advanced Scientific Computing Research. Additionally, JC and MG were supported by US Department of Energy grant DE-SC0009324 and US Air Force Office of Scientific Research grant FA9550-15-1-0001. }

\subjclass[2010]{Primary 	65N30}

\begin{document}

\begin{abstract}
{Meshing of geometric domains having curved boundaries by affine simplices produces a polytopial approximation of those domains. The resulting error in the representation of the domain limits the accuracy of finite element methods based on such meshes. On the other hand, the simplicity of affine meshes makes them a desirable modeling tool in many applications.  In this paper, we develop and analyze higher-order accurate finite element methods that remain stable and optimally accurate on polytopial approximations of domains with smooth boundaries. This is achieved by constraining a judiciously chosen extension of the finite element solution on the polytopial domain to weakly match the prescribed boundary condition on the true geometric boundary. We provide numerical examples that highlight key properties of the new method and that illustrate the optimal $H^1$ and $L^2$-norm convergence rates.}
\end{abstract}

\maketitle

%%%%%%%%%%%%%%%%%%%%%%%%%%%%% 
\section{Introduction}

It is well known that standard finite element methods based on piecewise polynomials of degree greater than one do not achieve optimal accuracy whenever a domain $\Omega$ having a curved boundary is approximated by a polygonal or polyhedral domain $\Omega_h$. Table \ref{table0} illustrates this fact for finite element approximations of a smooth solution of the Poisson equation on the unit disc approximated by inscribed regular polygons with sides of length $h$.
%%%%%%%%%%%%%%%
\begin{table}[h!]
\caption{\noindent Finite element convergence rates for smooth solutions of a Poisson equation on the unit disk approximated by a sequence of regular inscribed polygons with side length $h$. The last row shows the theoretical convergence rate of the best approximation (BA) out of each finite element space.
}\label{table0} 
\begin{center}
\begin{tabular}{|c||c|c||c|c||c|c|}
\hline
%\multicolumn{6}{|c|}{Element type}\\
Element type &
\multicolumn{2}{|c||}{Quadratic}&\multicolumn{2}{|c||}{Cubic}&\multicolumn{2}{|c|}{Quartic }\\
\hline  
Error type&
$L^2(\Omega_h)$ & $H^1(\Omega_h)$  &  $L^2(\Omega_h)$ & $H^1(\Omega_h)$  &   $L^2(\Omega_h)$ & $H^1(\Omega_h)$\\
\hline
Convergence rate&
2.188 & 1.698  & 2.115 & 1.590  & 2.151 & 1.590\\
\hline
BA rate & 3.0 & 2.0 & 4.0 & 3.0 & 5.0 & 4.0 \\
\hline
\end{tabular}
\end{center}
\end{table}
%%%%%%%%%%
The table shows that in all cases the $L^2(\Omega_h)$-norm convergence rate is capped at approximately $2$ whereas the $H^1(\Omega_h)$-norm convergence rate is approximately $3/2$. Of course, the explanation for such loss of precision is also well known: the approximation theoretic convergence rates for higher-degree polynomials are swamped by the geometric error of $O(h^2)$ resulting from defining the finite element discretization on the approximate domain $\Omega_h$ instead of the true domain $\Omega$, including imposing the boundary condition on the boundary $\Gamma_h$ of $\Omega_h$ instead of on the exact boundary $\Gamma$ of $\Omega$. This loss of accuracy is an example of a \emph{variational crime \em \cite[Chapter 4, p. 172]{strang:1973} \em and has nothing to do with the regularity of the exact solution}; indeed, the loss occurs for $C^\infty(\overline\Omega)$ and even analytic exact solutions.

In this paper, we develop and analyze a new finite element formulation that remains, under certain assumptions, optimally accurate for finite element spaces of arbitrary orders defined on polytopial approximations of geometric domains with smooth boundaries. The significance of this work stems from the fact that finite element methods based on affine simplicial grids remain one of the most efficient instances of this class of methods, both in terms of mesh generation and computational costs. For example, an affine simplex has a constant Jacobian determinant that can be precomputed, thereby allowing significant savings in the application of various pullbacks necessary for, e.g., compatible finite elements. Yet, because the resulting polytopial approximation of the geometric domain is only at best second-order accurate, such meshes create an accuracy bottleneck for higher-order elements. Overcoming this bottleneck is the main purpose of this paper.

To put our work in a proper context, we briefly discuss relevant mesh types and survey related existing literature.  

{\textbf{\em Simplicial mesh types.}}
Meshing of a domain $\Omega$ with curved boundaries by affine simplices yields a polytopial approximation $\Omega_h$ of the former, where $\Omega_h$ is the union of all the simplices. In many practical cases, all vertices on the approximate boundary $\Gamma_h$ lie on the exact boundary $\Gamma$. We refer to such meshes as {\em Type A meshes.} Alternately, for a \emph{Type B mesh}, none or at least not all of the vertices of $\Gamma_h$ lie on $\Gamma$. The simplest examples of Type A and B meshes are inscribed and circumscribed polygons for a disk, respectively. For a Type A mesh the distance between the boundaries of $\Omega$ and $\Omega_h$ is of $O(h^2)$, where $h$ is a measure of the size of the finite element grid cells. For Type B meshes, this distance can be larger than $O(h^2)$. In this work we restrict attention to Type A meshes and Type B meshes for which the distance between the discrete and continuous boundaries is of order $O(h^2)$.

{\textbf{\em Existing work.}}
There are two fundamentally different strategies for achieving optimal error bounds for high-order elements on curved domains. The first focuses on reducing the geometric approximation error in $\Omega_h$ without modifying the underlying variational formulation for the finite element method. A classical example of this idea is the isoparametric finite element method \cite{ergatoudis1968curved} that maps reference elements to curvilinear elements using polynomial transformations of the same degree as that of the finite element space. However, this approach increases the computational cost and, more importantly, only elements of order $k\leq 2$ are able to achieve optimal convergence with respect to the $H^1$ norm \cite{Fried_73_JSV}. For the special case of two dimensions and cubic elements, one can select nodes for which the finite element interpolant is optimally accurate  \cite{ciarlet1972interpolation} but, unfortunately, the finite element solution of the Poisson problem remains suboptimally accurate.

Another example of the first strategy is the isogeometric analysis approach (IGA) \cite{cottrell2009isogeometric, hughes2005isogeometric} that uses nonuniform rational B-splines (NURBS) as a finite element basis and achieves optimal accuracy for curved domains. IGA generates a mesh  of control points for the NURBS basis and then applies a transformation map to the control points to obtain a highly accurate approximation $\Omega_h$ of the curved domain $\Omega$. However, the NURBS basis makes the IGA approach more difficult to implement and more costly to solve than traditional polynomial-based finite elements.

The second, less explored strategy, retains the polytopial domain approximation but modifies the underlying variational problem in order to compensate for the fixed geometric error in $\Omega_h$. For example, optimal error estimates are obtained in \cite{cockburn2014priori, cockburn2012solving, cockburn2014solving} for Type B meshes by using polynomial extensions and line integrals to transfer boundary values from the curved boundary $\Gamma$ to the approximate boundary $\Gamma_h$. The primary difficulties of this approach include the construction of line integrals and the additional expense incurred because of the use of the hybridized discontinuous Galerkin method on mixed formulations of elliptic PDEs.

Recently, in \cite{MAIN2017,scovazziShiftedBoundary2}, a method was developed that achieves optimal error estimates for piecewise linear elements on Type B meshes for Dirichlet elliptic boundary-value problems. A linear extension is constructed to weakly match the boundary conditions by using the Nitsche method. Optimal $H^1$-norm convergence rates are demonstrated  even if the distance between the computational and the real boundaries is $O(h)$. The stability of this approach depends on the specific choices of stabilization parameters. Suboptimal convergence rates estimates are obtained with respect to the $L^2$ norm and higher-order finite element approximations and Neumann boundary conditions are not considered.

{\textbf{\em What is new in this paper.}}
Our new approach is an example of the second strategy, i.e., it relies on suitable modifications of the variational formulation when defining the finite element method in order to recover optimal convergence rates on polytopial approximations of curved domains. The method is applicable to both Dirichlet and Neumann boundary conditions. In a nutshell, it forces a polynomial extension of the approximate solution to match the prescribed boundary condition data on the boundary of the given domain $\Omega$; thus, we refer to this approach as the {\em polynomial extension finite element method} (PE-FEM). The extended Dirichlet condition is weakly enforced whereas the extended Neumann condition is enforced as a natural condition for a modified weak formulation of the boundary-value problem. We prove stability and optimal $H^1(\Omega_h)$ accuracy for both the Dirichlet and Neumann problems and show that, on convex meshes and under additional regularity assumptions, the former also converges optimally in $L^2(\Omega_h)$. Furthermore, computational studies indicate that optimal $L^2(\Omega_h)$-norm convergence is also achieved for the Neumann problem. In addition to recovering optimal accuracy, the method is computationally efficient and simple to implement.

The paper is organized as follows. Section \ref{section: preliminaries} introduces the necessary technical background. In \S\ref{section: PE-FEM method}, we describe the PE-FEM Dirichlet and Neumann formulations and then, in \S\ref{section: analysis}, we prove the well posedness of the discretized problems. Then, in \S\ref{section: analysis2}, we derive optimal error estimates with respect to the $L^2(\Omega_h)$ and $H^1(\Omega_h)$ norms for the Dirichlet problem and optimal $H^1(\Omega_h)$ norm error bounds for the Neumann problem. To streamline the flow of the paper, we relegate long proofs to the appendix. In \S\ref{section: implementation}, we discuss some implementation issues attendant to the PE-FEM and, in \S\ref{section: results}, we provide illustrative numerical results for the PE-FEM based on Type A meshes. Concluding remarks are provided in \S\ref{sec:conclusion}.

%%%%%%%%%%%%%%%%%%%%%%%%%%% 
\section{Preliminaries} \label{section: preliminaries}

Let $k\ = 2,3 \ldots$ and let $\Omega\subset\mathbb R^\DDD$, $\DDD = 2$, $3$, denote a bounded, open domain having a $C^{k+1}$ boundary $\Gamma$ with $\mathbf n$ denoting the outer unit normal vector. We consider approximations of $\Omega$ by affine simplicial meshes $\Omega_h$, i.e., collections of open $\DDD$-simplices $\{\KKK_j\}$ such that the non-empty intersections of their closures consist of only vertices, complete edges, or complete faces. Here $h := \max_{{\KKK_j} \in \Omega_h}\textrm{diam}({\KKK_j})$ denotes the mesh size parameter. Every mesh $\Omega_h$ defines a polytopial approximation of $\Omega$, which we also denote by $\Omega_h$; see Figure \ref{fig1} for a two-dimensional illustration. We note that the boundary $\Gamma_h$ of $\Omega_h$ is a union of $(\DDD-1)$-simplices $\{{\EEE_i}\}$ so that the outer unit normal vector $\mathbf n_h$ to $\Gamma_h$ is in general a piecewise constant vector and is thus only piecewise continuous. For every $\EEE_i \in \Gamma_h$, let $\KKK_{j_i}$ denote the element of $\Omega_h$ whose closure contains $\EEE_i$ on its boundary. Throughout, $C$ denotes a positive constant whose value changes from one instance to another but which does not depend on $h$.
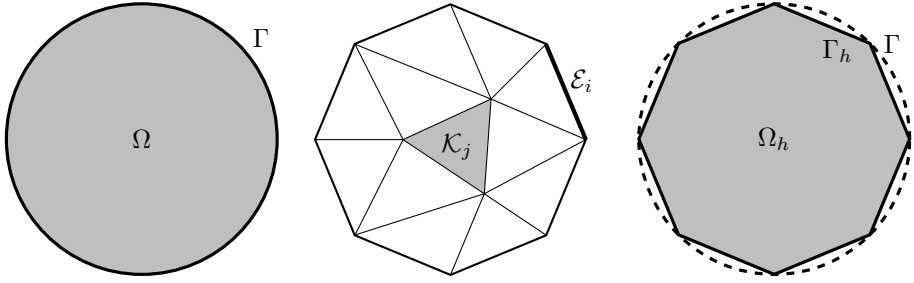
\begin{figure}[h!]
\begin{center}
	\begin{tikzpicture}[scale=0.9]
	\filldraw[color=black! 100, fill=lightgray, very thick](0,0) circle (2);
	\filldraw[black] (0,0) node[anchor=center] {$\Omega$};
	\filldraw[black] (1.5, 1.5) node[anchor=west]{$\Gamma$};
	\end{tikzpicture}
	\quad	
		\begin{tikzpicture}[scale=0.9]
	 \draw[xshift=0.0, thick] (0:2) \foreach \x in {45,90,...,359} {
                -- (\x:2)
            } -- cycle (90:2);
        \filldraw[color=black!100, fill=lightgray] (0.6, 0.6) -- (-0.7, 0) -- (0.5, -0.8) -- cycle;
        \draw[black] (0.6, 0.6) -- (0, 2);
        \draw[black] (0.6, 0.6) -- (1.41, 1.41);
        \draw[black] (0.6, 0.6) -- (2, 0);
        \draw[black] (0.5, -0.8) -- (2,0);
        \draw[black] (0.5, -0.8) -- (1.41, -1.41);
        \draw[black] (0.5, -0.8) -- (0, -2);
        \draw[black] (0.5, -0.8) -- (-1.41, -1.41);
        \draw[black] (-0.7, 0) -- (-1.41, -1.41);
        \draw[black] (-0.7, 0) -- (-2, 0);
        \draw[black] (-0.7,0) -- (-1.41, 1.41);
        \draw[black] (0.6, 0.6) -- (-1.41, 1.41);
	\filldraw[black] (0.1,-0.1) node[anchor=center]{$\KKK_j$};	
	\draw[black, ultra thick] (1.41, 1.41) -- (2,0);
	\filldraw[black] (1.94, 0.86) node[anchor=center]{$\EEE_i$};
	\end{tikzpicture}
	\quad
	\begin{tikzpicture}[scale=0.9]
	 \filldraw[xshift=0.0, fill=lightgray, very thick] (0:2) \foreach \x in {45,90,...,359} {
                -- (\x:2)
            } -- cycle (90:2);
        	\draw[color=black! 100, very thick, dashed](0,0) circle (2);
	\filldraw[black] (0,0) node[anchor=center] {$\Omega_h$};
	\filldraw[black] (1.3, 1.3) node[anchor=east]{$\Gamma_h$};
	\filldraw[black] (2.0, 1.4) node[anchor=east]{$\Gamma$};
	\end{tikzpicture}
\end{center}
	\caption{A curved domain $\Omega$ (left), an associated affine simplicial mesh $\Omega_h$ (center), and the resulting polygonal approximation $\Omega_h$ (right).}\label{fig1}
\end{figure}

It follows from the smoothness assumption made about $\Gamma$ that, for every ${\EEE_i} \in \Gamma_h$, there exists a $C^{k+1}({\overline {\EEE}_i})$ mapping 
$\etab_i: {\EEE_i} \rightarrow \Gamma$ such that  $\etab_i(\xib) \in \Gamma$ for every $\xib \in {\EEE_i}$ and such that
\begin{equation}\label{eqn: distance assumption}
{\max_{{\EEE_i}\in \Gamma_h} \sup_{\xib\in {\EEE_i}} |\etab_i(\xib) - \xib| \leq \deltah }
\end{equation}
for some $\deltah \in \mathbb R^+$, where $|\cdot|$ denotes the Euclidean norm. The mappings $\etab_i$ define a piecewise $C^{k+1}$ map $\etab: \Gamma_h\rightarrow \Gamma$. The value of  $\deltah$ in \eqref{eqn: distance assumption} can be viewed as a measure of the geometric error in the approximation of $\Omega$ by $\Omega_h$. See the left sketch in Figure \ref{fig2} for an illustration. The right sketch of Figure \ref{fig2} illustrates the pullback from $\Gamma$ to $\Gamma_h$, i.e., how the value of a function $v(\etab)$ evaluated at a point $\etab_i\in\Gamma$ is pulled back to the point $\xib\in{\EEE_i}\subset \Gamma_h$.
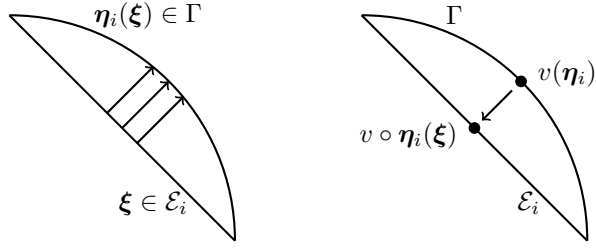
\begin{figure}[h!]
\begin{center}
	\begin{tikzpicture}
		\draw[black, thick] (0,3) -- (3, 0);
		\draw[thick] (3,0) arc (0:90:3);
		\draw[black, thick, ->] (1.5, 1.5) -- (2.12, 2.12);
		\draw[black, thick, ->] (1.3, 1.7) -- (1.92, 2.32);
		\draw[black, thick, ->] (1.7, 1.3) -- (2.32, 1.92);
		\filldraw[black] (2.5, 0.5) node[anchor = east]{$\xib \in {\EEE_i}$};
		\filldraw[black] (1, 3) node[anchor = west]{$\etab_i(\xib) \in \Gamma$};
	\end{tikzpicture}
	\qquad\qquad
	\begin{tikzpicture}
		\draw[black, thick] (0,3) -- (3, 0);
		\filldraw[black] (1.5, 1.5) circle (2pt);
		\filldraw[black] (1.4, 1.4) node[anchor = east]{$v\circ\etab_i(\xib)$};
		\draw[thick] (3,0) arc (0:90:3);
		\filldraw[black] (2.12, 2.12) circle(2pt);
		\filldraw[black] (2.22, 2.22) node[anchor = west]{$v(\etab_i)$};
		\draw[black, thick, <-] (1.6, 1.6) -- (2.0, 2.0);
		\filldraw[black] (2.5, 0.5) node[anchor = east]{${\EEE_i}$};
		\filldraw[black] (1, 3) node[anchor = west]{$\Gamma$};
	\end{tikzpicture}
\caption{Left: An example of a $C^{k+1}$ mapping $\etab_i: {\EEE_i}\rightarrow \Gamma$ defined as the intersection of a line normal to ${\EEE_i}$ with the true boundary $\Gamma$. Right: A sketch of a pullback from the continuous boundary onto the polygonal boundary.}\label{fig2}
\end{center}
\end{figure} 

Let $\alphab = (\alpha_n)_{n=1}^\DDD$, $\alpha_n$ a non-negative integer, denote a multi-index and let $|\alphab| = \sum_{n=1}^\DDD \alpha_n$ and $\alphab! = \prod_{n=1}^\DDD \alpha_n!$. For $\mathcal D=\Omega$ or $\Omega_h$ and for $m\in{\mathbb N}$, let $H^m(\mathcal D)$ denote the standard Sobolev space and $(H^m(\mathcal D))'$ the corresponding dual space; see \cite{adams2003sobolev}. Also, for any $\xib\in\mathbb R^\DDD$, let $\xib^\alphab :=\xi_1^{\alpha_1}\xi_2^{\alpha_2}\cdots\xi_\DDD^{\alpha_\DDD}$ and $D^\alphab :=\partial^{|\alphab|}/\partial^{\alpha_1}\xi_1\partial^{\alpha_2}\xi_2\cdots\partial^{\alpha_\DDD}\xi_\DDD$. For $\mathcal D=\Gamma$ or $\Gamma_h$, we consider the fractional Sobolev space $H^{m-\frac12}(\mathcal D)$. The $k$-th order Lagrange finite element space is defined by
\begin{equation*}
	V^k_h := \left\{ v\in C^0(\overline\Omega_h)\,\,:\,\, v|_{\KKK_j} \in P_k(\KKK_j) \quad \forall\, \KKK_j \in \Omega_h \right\}\subset H^1(\Omega_h),
\end{equation*}
where $P_k(\KKK_j)$ denotes the space of polynomials of degree at most $k$ defined over the $\DDD$-simplex $\KKK_j {\subset}  \mathbb R^\DDD$. In addition, we have the constrained space
\begin{equation*}
	V^k_{h,0} := \left\{ v \in V^k_h \,\,:\,\, v = 0 \onn \Gamma_h \right\} \subset H^1_0(\Omega_h)
\end{equation*}
and the trace space 
\begin{equation*}
	W^k_h:= V^k_h\big|_{\Gamma_h}= \left\{v \in C^0(\Gamma_h)\,\,:\,\, v|_{\EEE_i} \in P_k(\EEE_i) \quad \forall\, \EEE_i \in \Gamma_h\right\}
	  \subset H^{1/2}(\Gamma_h).
\end{equation*}
We also define the discontinuous finite element space
$$
\overline V^k_{h} := \left\{ v \in L^2(\Omega_h) :\,\, v|_{\KKK_j} \in P_k(\KKK_j) \quad \forall\, \KKK_j \in \Omega_h \right\}
$$
and the discrete differential operator $D_h^{\alpha} : \overline V^k_h \rightarrow L^2(\overline \Omega_h)$ defined by 
$$
D_h^{\alpha} v_h(\xb) := 
\begin{cases}
	D^\alpha v_h(\xb) &\textrm{ if } \xb \in {\KKK_j} \text{ for } \KKK_j \in \Omega_h\\
	0 &\textrm{ otherwise}.
\end{cases}
$$
Duality pairings over $\Omega_h$ and $\Gamma_h$ are defined by
\begin{equation*}
	\langle v, w \rangle_{\Omega_h} = \sum_{\KKK_j\in \Omega_h} \int_{\KKK_j} vw d\KKK_j\quad
	\textrm{and}\quad
	\langle v, w \rangle_{\Gamma_h} = \sum_{{\EEE_i} \in \Gamma_h} \int_{{\EEE_i}} vw d{\EEE_i},
\end{equation*}
respectively.
``Broken'' Sobolev norms on $\Omega_h$ and $\Gamma_h$ are defined by 
\begin{equation*}
	\trino v \trino^2_{m, \Omega_h} = \sum_{\KKK_j\in \Omega_h} \| v \|_{m, \KKK_j}^2 \,\,\,\,
	\forall\, v\in V^k_{h}
	\quad\textrm{and} \quad
	\trino w \trino^2_{m, \Gamma_h} = \sum_{{\EEE_i} \in \Gamma_h} \|w \|_{m, {\EEE_i}}^2
	\,\,\,\, \forall\, w\in W^k_{h},
\end{equation*}
respectively. On the discrete spaces $V^k_h$ and $W^k_h$ we have the inverse inequalities involving the corresponding ``broken'' norms given by
\begin{equation*}
	\trino v \trino_{m, \Omega_h} \leq C h^{-1} \trino v \trino_{m-1, \Omega_h}
\,\,\,\,
	\forall\, v\in V^k_{h},\; m=1,2,\ldots
\end{equation*} 
and 
\begin{equation*}
	\trino w \trino_{m+1/2, \Gamma_h} \leq Ch^{-\frac12} \trino w \trino_{m, \Gamma_h} \, \, \, \, \forall w\in W^k_{h},\; m=0,1,\ldots.
\end{equation*}

The smoothness assumption on $\Gamma$ implies the existence of a continuous lifting operator $\mathcal R(\cdot): H^{k+1/2}(\Gamma) \rightarrow H^{k+1}(\Omega)$ such that for all $g\in H^{k+1/2}(\Gamma)$ there exists $v=\mathcal R(g) \in H^{k+1}(\Omega)$ with $\| v\|_{k+1,\Omega}\le C_{\mathcal R}\|g\|_{k+\frac12,\Gamma}$. We also have the continuous discrete lifting operator $\mathcal R_h(\cdot): W^k_h \rightarrow V^k_h$ such that for all $g_h\in W^k_h$ there exists $v_h=\mathcal R_h(g_h) \in V^k_h$ with $\| v_h\|_{k+1,\Omega_h}\le C_{\mathcal R_h}\|g\|_{k+\frac12,\Gamma_h}$. 

Finally we recall the approximation theoretic bound
\begin{equation}\label{eqn: best approximation bound}
	\inf_{\chi \in V^k_h}\| v - \chi \|_{s, \Omega_h} \leq C h^{k-s+1} |v|_{k+1, \Omega_h}
	\quad\mbox{for $s=0,1$ and $\forall\, v\in H^{k+1}({\Omega_h})$}
\end{equation}
that holds under the assumption that $\Omega_h$ is a regular mesh \cite{ciarlet2002finite}.

%%%%%%%%%%%%%%%%%%%%%%%%%
\subsection{Setting} \label{section:problem}

To present the key ideas of the method without unnecessary technical complications
 we consider the Dirichlet problem 
\begin{equation}\label{contprobd}
	-\nabla \cdot \big( p({\bf x}) \nabla u \big)   = f \quad \inn \Omega \qquad\mbox{and}\qquad u = g_D \quad\onn \Gamma
\end{equation}
and the Neumann problem
\begin{equation}\label{contprobn}
	-\nabla \cdot \big( p({\bf x}) \nabla { u} \big) + q({\bf x}) { u} = f \quad \inn \Omega \qquad\mbox{and}\qquad p\nabla u \cdot \mathbf n = g_N\quad \onn \Gamma .
\end{equation}
Here, $p, q\in C^k(\overline\Omega)$, $g_D({\bf x})\in H^{k+1/2}(\Gamma)$, $g_N({\bf x})\in H^{k-1/2}(\Gamma)$,   and $f\in H^{k-1}(\Omega)$ are given functions such that $\underline{p}\le p({\bf x})\le\overline{p}$ for some $\underline{p}>0$ and $\overline{p}<\infty$ and $q({\bf x})>0$.\footnote{The last assumption obviates the need to work in the quotient space $H^1(\Omega)\setminus\mathbb R$ for the Neumann problem.}

A weak formulation of \eqref{contprobd} seeks $u \in H^1(\Omega)$ such that
\begin{equation}\label{eqn: weak Dirichlet}
	D(u,v) = \left<f,\,v\right>_{\Omega} \quad \forall\, v \in H^1_0(\Omega)\qquad\mbox{and}\qquad u = g_D \onn \Gamma
\end{equation}
whereas a weak formulation of \eqref{contprobn} seeks $u \in H^1(\Omega)$ such that
\begin{equation}\label{eqn: weak Neumann}
	N(u,v) = \left<f,\,v\right>_{\Omega}  + \langle g_N, v\rangle_{\Gamma} \quad \forall\, v \in H^1(\Omega),
\end{equation}
where  the bilinear forms $D(\cdot, \cdot): H^1(\Omega)\times H^1(\Omega)\rightarrow \mathbb R$ and $N(\cdot, \cdot): H^1(\Omega)\times H^1(\Omega)\rightarrow \mathbb R$ are defined by
\begin{equation*}
	D(u,v) := \int_{\Omega}  p \nabla u\cdot \nabla v \,d{\bf x}\qquad\mbox{and}\qquad
	N(u,v) := \int_{\Omega} ( p \nabla u\cdot \nabla v +quv) \,d{\bf x},
\end{equation*}
respectively. Both \eqref{eqn: weak Dirichlet} and \eqref{eqn: weak Neumann} are well-posed for $f \in H^{-1}(\Omega)$, $g_D \in H^{1/2}(\Gamma)$, and $g_N \in H^{- 1/2}(\Gamma)$, whereas our regularity assumptions on $\Gamma$, $g_D$, $g_N$, $p, q,$ and $f$ imply that $u\in H^{k+1}(\Omega)$. 

In general, $\Omega_h\not\subset\Omega$ and $\Omega\not\subset\Omega_h$; see Figure \ref{fig1b} for an illustration. As a result, { if $\Omega_h\not\subset\Omega$, the data $p$, $q$, and $f$ and the solution $u$ of \eqref{contprobd} or \eqref{contprobn} may not be defined on all of $\Omega_h$  so that extensions of these functions from $\Omega$ to $\Omega\cup\Omega_h$ are required.} Our regularity assumptions imply the existence of bounded extensions\footnote{The existence of $C^k$ extensions for the problem coefficients is a consequence of the Tietze-Urysohn extension theorem \cite{ciarlet2013linear}. The existence of bounded extensions $\widetilde f\in H^{k-1}(\mathbb R^\DDD)$ for $k=2,3,\ldots$ is a classical result of Sobolev spaces \cite{adams2003sobolev}. For $k=1$, we can construct $\widetilde f$ by extending $f$ to zero outside of $\Omega$. For $k=0$, we can construct $\widetilde f_n$ as follows. Because $L^2(\Omega)$ is dense in $H^{-1}(\Omega)$, we can write $f$ as the limit of a sequence $f_n \in  L^2(\Omega)$. We construct the functions $\widetilde f_n \in L^2(\mathbb R^\DDD)$ by extending the functions $f_n$ to zero outside of $\Omega$. We note that $\widetilde f_n$ is a Cauchy sequence in $H^{-1}(\mathbb R^\DDD)$ because  $\|\widetilde f_n - \widetilde f_m\|_{-1, \mathbb R^\DDD} \le \|f_n-f_m\|_{-1, \Omega}$ (recall that $\|u\|_{-1, \Omega} := \sup \limits_{v \in H^1(\Omega)} (u,v)_{\Omega}/ \|v\|_{1,\Omega}$). Therefore $\widetilde f_n$ is convergent in $H^{-1}(\mathbb R^\DDD)$ and we define $\widetilde f$ to be its limit.}
$\widetilde p\in C^k(\mathbb R^\DDD)$, 
$\widetilde q\in C^k(\mathbb R^\DDD)$, 
$\widetilde f\in H^{k-1}(\mathbb R^\DDD)$, and 
 $\tu \in H^{k+1}(\mathbb R^\DDD)$
such that $\widetilde p = p$, $\widetilde q = q$, $\widetilde u = u$ and, for $k \ge 1$, $\widetilde f = f$ almost everywhere in $\Omega$. For $k=0$, we have $\langle\widetilde f, v\rangle_{\Omega} = \left<f, v\right>_{\Omega}$ for all $ v \in H^1(\Omega)$.
In particular, there exist extensions such that $\|\widetilde u\|_{k+1,\Omega\cup\Omega_h}\le C_e\|u\|_{k+1,\Omega}$, $\|\widetilde f\|_{k-1,\Omega\cup\Omega_h}\le C_e\|f\|_{k-1,\Omega}$, $\|\widetilde p\|_{C^k({\overline{\Omega\cup\Omega_h})}}\le C_e\|p\|_{C^k({\overline{\Omega}})}$, and $\|\widetilde q\|_{C^k(\overline{\Omega\cup\Omega_h})}\le C_e\|q\|_{C^k(\overline{\Omega})}$ for a constant $C_e>0$ having value independent of $u$, $f$, $p$, or $q$.  

\begin{figure}[h!]
\begin{center}
\begin{tikzpicture}
\draw[color=black! 100, very thick](0,0) circle (3);
\filldraw[xshift=0.0, fill=lightgray, very thick] (0:3) \foreach \x in {45,90,...,359} {
                -- (\x:3)
         } -- cycle (90:2);
\filldraw[color=black! 100, fill=gray, very thick] (0,0) circle(1.5);
\filldraw[xshift=0.0, fill=white, very thick] (0:1.5) \foreach \x in {45,90,...,359} {
                -- (\x:1.5)
         } -- cycle (90:2);
\end{tikzpicture}
\end{center}
\caption{The area between the concentric circles is the given domain $\Omega$ and the area between the concentric octagons is the approximate domain $\Omega_h$ (the regions covered by the two shades of gray). The light gray region is $\Omega\cap\Omega_h$. The dark gray regions are in $\Omega_h$ but not in  $\Omega$ so extensions of functions defined on $\Omega$ are needed in those regions.}\label{fig1b}
\end{figure}
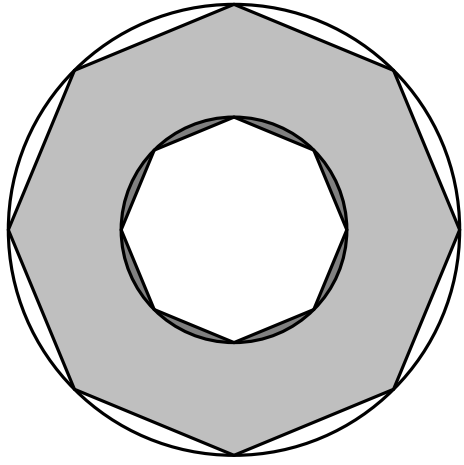

%%%%%%%%%%%%%%%%%%%%%%%
\section{The PE-FEM method}\label{section: PE-FEM method}

We now introduce the \emph{Polynomial Extension Finite Element Method} (PE-FEM) for the approximate solution of \eqref{eqn: weak Dirichlet} or \eqref{eqn: weak Neumann} defined on polytopial domains $\Omega_h$ resulting from meshing of the domain $\Omega$ by affine simplices. To achieve optimal accuracy even if $\Omega$ has a curved boundary, we force the extension of the finite element solution to weakly match the given data on the curved boundary of the continuous problem.

We define the bilinear forms $\WTD(\cdot, \cdot): H^1(\Omega_h)\times H^1(\Omega_h)\rightarrow \mathbb R$ and $\WTN(\cdot, \cdot): H^1(\Omega_h)\times H^1(\Omega_h)\rightarrow \mathbb R$ as 
\begin{equation}\label{eqn: discrete bilinear definitions1}
\begin{aligned}
	{\WTD}(u, v) &:= \int_{\Omega_h}\widetilde p\,\nabla u \cdot \nabla v  \,d{\bf x}\qquad\forall\, u,v\in H^1(\Omega_h)\\
	{\WTN}(u, v) &:= \int_{\Omega_h}\big( \widetilde p\, \nabla u \cdot \nabla v +  \widetilde q\, u v\big) \,d{\bf x}\qquad\forall\, u,v\in H^1(\Omega_h),\\
\end{aligned}
\end{equation}
respectively.

We start presenting the PE--FEM method for the Dirichlet and Neumann problems in its simplest form, and then, in the rest of the section, we will derive the method in an equivalent formulation that is more amenable to being mathematically analyzed.

%%%%%%%%%%%%%%%%
{\textbf{\em The PE-FEM for the Dirichlet problem.}}
Seek $u_h \in V^k_h$ such that
\begin{equation}\label{eqn: PEF Dirichlet problem0}
	{\WTD}(u_h, v) = \langle\widetilde f, \,v\rangle_{\Omega_h} \quad \forall\, v \in V^k_{h,0}
\end{equation}
and
\begin{equation}\label{eqn: PEF Dirichlet problem0d}
	\sum_{i} \big\langle E_{\KKK_{j_i}} (u_h) \circ \etab(\xib), \mumu \big\rangle_{\EEE_i} = \big\langle g_D\circ\etab(\xib), \mumu \big\rangle_{\Gamma_h} 
	\quad \forall\, \mumu \in  W_h^k = V^k_h\big|_{\Gamma_h},
\end{equation}
where $E_{\KKK_{j_i}} (u_h)$ is the operator that extends the polynomial $u_h|_{\KKK_{j_i}}$ to a polynomial over $\mathbb R^\DDD$. Implementation of this method requires the evaluation of $E_{\KKK_{j_i}} (u_h) \circ \etab(\xib)$ at a set of quadrature points $\{\xib_q\}$ { on $\EEE_i$}. This can be accomplished by evaluating the polynomial basis functions that generate $u_h|_{\KKK_{j_i}}$ at the points $\{\etab(\xib_q)\}$ that can be outside the element $\KKK_{j_i}$.

{\textbf{\em The PE-FEM for the Neumann problem.}}
Seek $u_h \in V^k_h$ such that
\begin{equation}\label{eqn: PEF Neumann problem0}
{\WTN}(u_h, v) + \boldsymbol{\tau}_{N} (u_h,v) =  \langle\widetilde f, \,v\rangle_{\Omega_h} + \langle g_N \circ \etab(\xib), v \rangle_{\Gamma_h} \quad \forall\, v \in V^k_h,
\end{equation}
where
\begin{equation}\label{eqn: Neumann bilinear form0}
\boldsymbol{\tau}_{N} (u_h,v) := 
	\sum_i \Big\langle  \widetilde p\circ\etab(\xib)\;   \nabla \left( E_{\KKK_{j_i}} (u_h) \circ \etab(\xib) \right) \cdot \mathbf  n -  \widetilde p\,\nabla u_h \cdot \nn_h, v \Big\rangle_{\EEE_i} 
\end{equation}
is an auxillary term that provides additional accuracy to the standard finite element formulation. When $\Gamma_h = \Gamma$ and $\etab$ is taken to be the identity operator, we have that $\boldsymbol{\tau}_{N} (u_h,v) = {0}$. 

Additional details about the implementation of the method are provided in Section \ref{section: implementation}.

In order to analyze the method, we reformulate the PE-FEM problem so that it is well defined also for trial functions in $H^k(\Omega_h)$. In particular we need to generalize the extension operator $E$ to functions in Sobolev spaces. We achieve this by using averaged Taylor polynomials. The reformulated Dirichlet and Neumann PE-FEM problems will be equivalent to the ones in \eqref{eqn: PEF Dirichlet problem0}-\eqref{eqn: Neumann bilinear form0} whenever $u_h \in V^k_h$.

%%%%%%%%%%%%%%%%%%%%%%%%%%
\subsection{Averaged Taylor polynomial extensions} \label{sec: Taylor approximation}

The mismatch between the exact domain $\Omega$ and its polygonal approximation $\Omega_h$ requires approximation of the boundary condition data on $\Gamma_h$. In this section, we focus on the extension of functions belonging to $H^{k+1}(\Omega_h)$ from the approximate boundary $\Gamma_h$ onto the true boundary $\Gamma$. {A common approach is to approximate the true boundary condition data on $\Gamma_h$ by a low-order reconstruction. However, due to the geometric error resulting from the approximation of the domain, this approach restricts the numerical solution to be at best second-order accurate regardless of the degree of the underlying finite element space.} Here, instead of interpolating the boundary condition data, we \emph{extend} the finite element solution from the approximate boundary $\Gamma_h$ to the true boundary $\Gamma$ and require it to weakly match the boundary condition data prescribed on that boundary. The main tool we use for defining the extension is the averaged Taylor polynomials, described below.

For every $\EEE_i \in \Gamma_h$, let $\KKK_{j_i}$ denote the element of $\Omega_h$ whose closure contains $\EEE_i$ and let $\{S^{i,\mmm}\}$ denote a family, indexed by $\mmm$, of disjoint star-shaped domains with respect to the balls $\sigma^{i,\mmm} \subset \KKK_{j_i} \cap \Omega$ such that $S^{i,\mmm} \cap \KKK_{j_{i'}} = \emptyset$ for $i \neq i'$, $\mbox{diam}(S^{i,\mmm}) \le C \deltah$, and $\overline{ {\cup_{i,\mmm}} S^{i,\mmm}} \supset \overline{\Omega_h \setminus (\Omega \cap \Omega_h)}$. We also require that $\overline {S^{i,\mmm}} \cap \etab(\EEE_i) \subset  S^{i,\mmm}$ and $\overline {S^{i,\mmm}} \cap \EEE_i \subset S^{i,\mmm}$ and that $\sup \limits_{i,\mmm} \frac{\mbox{diam}(S^{i,\mmm})}{\mbox{radius}(\sigma^{i,\mmm})} \le C$ with $C$ independent of ${\deltah}$. See Figure \ref{kite} for an illustration of how star-shaped domains $S^{i,\mmm}$ can be constructed for triangular meshes.

\begin{remark}
Whereas it possible to construct the star-shaped domains $S^{i,\mmm}$ with the properties listed above for simple geometries/meshes, we do not have a proof for  general domains and shape-regular meshes considered in this paper. If we allow the domains $S^{i,\mmm}$ to overlap up to a finite number of times, it may be possible to follow the construction in {\em\cite[Section 2.2]{narcowich_sobolevBounds_2005}.} However, this would further increase the complexity of the analysis so we prefer to limit our analysis to the case for which the domains $S^{i,\mmm}$ do not overlap.
\end{remark}

\begin{figure}[h!]
\begin{center}
\begin{tikzpicture}
\draw[red, very thick, fill= red!10!white] (1.5, 2) --(2.15, 2.65) -- (2.65, 2.15) -- (2.0,1.50)  (1.5, 2)  arc (135: 315: 0.3525);
\draw[blue, dashed] (0,4) .. controls (1.65, 3.1141) and (3.1141, 1.65) .. (4,0);
\draw[magenta, dashed] (1.75,1.75)  circle (0.34); 
\draw[black, thick] (0,0) -- (0,4);
\draw[black, thick] (0,0) -- (4,0);
\draw[black, thick] (0,4) -- (4,0);
\filldraw[black] (1.75, 1.75) node[anchor = center]{\color{magenta} \tiny $\sigma^{i,\mmm}$};
\filldraw[black] (1, 1) node[anchor = center]{\color{black} \Large $\KKK_{j_i}$};
\filldraw[black] (2.75, 0.65) node[anchor = center]{\color{black} \Large $\EEE_i$};
\filldraw[black] (1,3.5) node[anchor = south]{\color{blue}\Large$\Gamma$};
\filldraw[black] (3.25, 2.1) node[anchor = center]{\color{red} \Large $S^{i,\mmm}$};
\draw[black, thick, ->] (2.8, 2.1) -- (2.25, 2.1);
\end{tikzpicture}
\qquad \qquad 
\begin{tikzpicture}
\draw[red, very thick, fill= red!10!white] (1.5 - 0.45, 2- 0.45) --(2.15- 0.3, 2.65- 0.3) -- (2.65- 0.3, 2.15- 0.3) -- (2.0- 0.45,1.50- 0.45)  (1.5- 0.45, 2- 0.45)  arc (135: 315: 0.3525);
\draw[blue, dashed] (0,4) .. controls (1.65 - 1.0, 3.1141 - 1.0) and (3.1141- 1.0, 1.65 - 1.0) .. (4,0);
\draw[magenta, dashed] (1.75- 0.45,1.75- 0.45)  circle (0.34); 
\draw[black, thick] (0,0) -- (0,4);
\draw[black, thick] (0,0) -- (4,0);
\draw[black, thick] (0,4) -- (4,0);
\filldraw[black] (1.75 - 0.45, 1.75 - 0.45) node[anchor = center]{\color{magenta} \tiny $\sigma^{i,\mmm}$};
\filldraw[black] (1 - 0.35, 1 - 0.35) node[anchor = center]{\color{black} \Large $\KKK_{j_i}$};
\filldraw[black] (2.75+ 1.0, 0.65) node[anchor = center]{\color{black} \Large $\EEE_i$};
\filldraw[black] (1,2.25) node[anchor = south]{\color{blue}\Large$\Gamma$};
\filldraw[black] (3.25, 2.1) node[anchor = center]{\color{red} \Large $S^{i,\mmm}$};
\draw[black, thick, ->] (2.8, 2.1) -- (1.75, 1.75);
\end{tikzpicture}
\end{center}
\caption{Illustration of the construction of a star-shaped (with respect to $\sigma^{i,\mmm}$) set $S^{i,\mmm}\subset \mathbb R^\DDD$ for $\Gamma \cap \KKK_{j_i} = \emptyset$ (left) and  $\Gamma \cap \KKK_{j_i} \neq \emptyset$ (right).}\label{kite}
\end{figure}
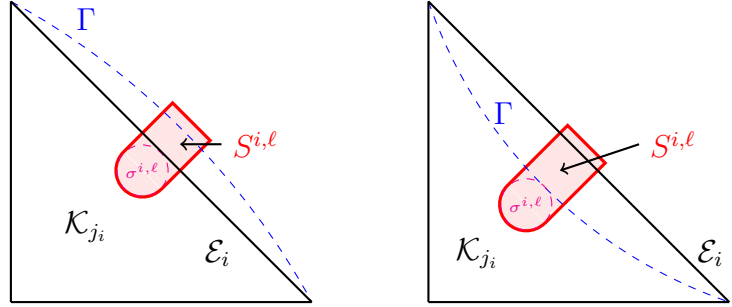

Following \cite{brenner2007mathematical}, we define, for $\xb \in \mathbb R^\DDD$ and $v \in L^2(\Omega \cap \Omega_h)$, the averaged Taylor polynomial\footnote{Averaged Taylor polynomials on star-shaped domains $S^{i,\mmm}$ are defined for functions in $L^1(\sigma^{i,\mmm})$; see \cite[Corollary 4.1.15.]{brenner2007mathematical}.}
\begin{equation}\label{Taylor definition}
\left . T^{k}_h (v)\right |_{\xb} := \sum_{i,\mmm} \mathbf 1_{S^{i,\mmm}}(\xb) \displaystyle \int_{\sigma^{i,\mmm}} {\Big( \sum^k_{|\alphab| = 0} \frac{1}{\alphab !} D^\alphab v(\yb) (\xb - \yb)^\alphab  \phi_\mmm(\yb) \Big) d \yb, }
\end{equation}
where $\phi_\mmm(\yb)$ is a cutoff function and $\mathbf 1_{S^{i,\mmm}(\xb)}$ denotes the indicator function for the set $S^{i,\mmm}$, i.e., $\mathbf 1_{S^{i,\mmm}(\xb)}= 1$ if $ \xb \in S^{i,\mmm}$ and vanishes otherwise. Note that $T^k_h$ is meaningful only for $\xb \in \bigcup_{i,\mmm} S^{i,\mmm}$ and is zero otherwise. For any $\xib \in \Gamma_h$ and its image $\etab(\xib) \in \Gamma$ and for ${v \in L^2(\overline \Omega)}$, we write 
\begin{equation}\label{Taylor on Gamma_h}
	 v\circ\etab(\xib) = \left . T^k_{h} (v) \right |_{\etab(\xib)} + \left . R^k_h (v) \right|_{\etab(\xib)}.
\end{equation}
For $v \in H^{k+1}(\mathbb R^\DDD)$ we have $\|R^k_h (v) |_{\etab(\xib)}\|_{0, \Gamma_h} \le C \deltah^{k+\frac12} |v|_{k+1, \mathbb R^\DDD}$; see Lemma \ref{lm: Taylor approximation}.
If $v \in \overline V^{k}_h$, then in every $\KKK_{j_i}$, $v$ is a polynomial of degree $k$ and therefore $T^k_h (v)$ exactly reproduces  $v$ in any $\KKK_{j_i}$ adjacent to the boundary and is equivalent to the classical Taylor polynomial. For $v\in \overline V^{k}_h$ we can therefore write, for a generic $\yb_i \in \KKK_{j_i}$,
$$
\begin{aligned}
\left . T^{k}_{h}(v) \right |_{\xb} &= \sum_{i,\ell} \mathbf 1_{S^{i,\mmm}}(\xb) \sum^k_{|\alphab| = 0} \frac{1}{\alphab !} D^\alphab v(\yb_i)(\xb_i - \yb_i)^\alphab \\
&= \sum_i \mathbf 1_{\left(\cup _j S^{i,\mmm}\right)}(\xb) \sum^k_{|\alphab| = 0} \frac{1}{\alphab !} D^\alphab v(\yb_i)(\xb_i - \yb_i)^\alphab.
\end{aligned}
$$
Setting $\yb_i = \xib \in \mathcal{E}_i$ and $\xb = \etab(\xib)$ we now have that
\begin{equation}\label{Taylor discrete}
\left . T^{k}_{h}(v) \right |_{\etab(\xib)} = \sum^k_{|\alphab| = 0} \frac{1}{\alphab !} D_h^\alphab v(\xib)(\etab(\xib) - \xib)^\alphab
\end{equation}
which is well-defined for any $\xib \in \Gamma_h$ and $v\in \overline V_h^{k}$. Note that if $\xib \in \EEE_i$ and $v$ is a polynomial of degree $k$ defined on $\KKK_{j_i}$ then $\left .T^{k}_{h}(v) \right |_{\etab(\xib)} \equiv E_{\KKK_{j_i}}(v) \circ \etab(\xib)$.
For convenience, we also define $T^{k',k}_{h}$ as
\begin{equation}\label{Taylor discrete m-k}
\left . T^{k',k}_{h}(v) \right |_{\etab(\xib)} = \sum^k_{|\alphab| = k'} \frac{1}{\alphab !} D_h^\alphab v(\xib)( \etab(\xib) - \xib)^\alphab.
\end{equation}
Clearly $T^{k}_{h} = T^{0,k}_{h}$. For vector functions $\mathbf v$, we introduce the vector operator $\mathbf{T}^{k}_{h}(\mathbf v) = \left(T^k_h v_n \right)_{n=1}^\DDD$. We use this notation in particular for gradients of scalar functions, i.e., $\mathbf{T}^{k}_{h}(\nabla v)$.

%%%%%%%%%%%%%%%%%%%%

%%%%%%%%%%%%%%%%%%%%%%%%%%%%

\subsection{PE-FEM methods using averaged Taylor polynomials} \label{sc:derivation}

Using the averaged Taylor polynomial extensions and Taylor's theorem allows one to represent the Dirichlet and Neumann data prescribed at $\etab(\xib)\in \Gamma$ as functions of $\xib\in\Gamma_h$ given by
\begin{equation}\label{eqn: Dirichlet representation}
	g_D \circ \etab(\xib) =  \left . T^k_h(\tu) \right |_{\etab(\xib)} + \left.R^k_{h} (\tu)\right|_{\etab(\xib)}
\end{equation}
and
\begin{equation}\label{eqn: Neumann representation}
	g_N \circ \etab(\xib) = \widetilde p\circ\etab(\xib)  \Big(  \mathbf{T}_h^{k-1}(\nabla \tu) \big|_{\etab(\xib)} \cdot \mathbf n
	 + \mathbf{R}^{k-1}_{h} (  \nabla \tu)\big|_{\etab(\xib)}\cdot \mathbf n \Big),
\end{equation}
respectively, where $ R^k_h (\tu) |_{\etab(\xib)}$ and $ \mathbf{R}^{k-1}_{h} \left(  \nabla \tu\right)|_{\etab(\xib)}$ denote the remainder terms of the averaged Taylor polynomials. These representations are used in the definition of the boundary conditions for the PE-FEM formulations.

{\textbf{\em The PE-FEM Dirichlet problem.}}
We use \eqref{eqn: Dirichlet representation} to supply the Dirichlet boundary condition \eqref{eqn: PEF Dirichlet problem0d} for the problem posed on the approximate domain $\Omega_h$. Note that for $u_h\in V_h^k$, the remainder term in \eqref{eqn: Dirichlet representation} vanishes. Then, for the Taylor polynomal extension approach, the \emph{PE-FEM Dirichlet problem} \eqref{eqn: PEF Dirichlet problem0} and \eqref{eqn: PEF Dirichlet problem0d} is to seek $u_h \in V^k_h$ such that
\begin{equation}\label{eqn: PEF Dirichlet problem}
\left\{\begin{aligned}
	{\WTD}(u_h, v) = \langle\widetilde f, \,v\rangle_{\Omega_h} &\quad \forall\, v \in V^k_{h,0}\\
	\big< T^k_h u_h(\xib)\big|_{\etab(\xib)}, \mumu \big>_{\Gamma_h} = \langle g_D\circ\etab(\xib), \mumu \rangle_{\Gamma_h} 
	&\quad \forall\, \mumu \in  W_h^k = V^k_h\big|_{\Gamma_h}.
\end{aligned}\right.
\end{equation}

\begin{remark}
The problem \eqref{eqn: PEF Dirichlet problem} is not a Dirichlet problem, per se. The boundary condition, i.e., the second equation in \eqref{eqn: PEF Dirichlet problem}, involves derivatives of the unknown $u_h$ of order up to $k$ evaluated along the boundary edges ${\EEE}_i$ of the approximate domain $\Gamma_h$. The inclusion of these derivatives in the boundary condition { imposed on the approximate boundary $\Gamma_h$} is, of course, what leads to the optimal accuracy of the PE-FEM approximation.
\end{remark}

In order to use the same space for the trial and test functions, we reformulate the problem \eqref{eqn: PEF Dirichlet problem} as follows. Let $(\cdot)_\star: V^k_h\rightarrow W^k_h$ denote the trace operator and $\mathcal R_h: W^k_h\rightarrow V^k_h$ a discrete linear lifting operator. Also, let 
\begin{equation}\label{eqn: Dirichlet bilinear form}
  B_{h,D}^{\scal}(u, v) := {\WTD}(u,v - \mathcal R_h v_\star) + \scalh \big< T^k_h u(\xib) \big|_{\etab(\xib)}, v \big>_{\Gamma_h}
\end{equation}
and
\begin{equation}\label{eqn: Dirichlet rhs term}
F_{h,D}^{\scal}(v) := \big< \widetilde f, \,v- \mathcal R_h v_\star\big>_{\Omega_h} + \scalh \langle g_D \circ \etab(\xib), v \rangle_{\Gamma_h}.
\end{equation}
We then seek $u_h \in V^k_h$ such that
\begin{equation}\label{eqn: PEF Dirichlet coercive problem}
B_{h, D}^\scal(u_h,v) = F_{h,D}^\scal(v)
	\quad \forall v\in V^k_h.
\end{equation}
Because $v - \mathcal R_h v_\star$ spans the entirety of $V_{h,0}^k$, the formulations \eqref{eqn: PEF Dirichlet problem} and \eqref{eqn: PEF Dirichlet coercive problem} are equivalent for any nonzero $\scalh \in \mathbb R$. {The choice of scaling factor $\scalh$ and does not affect the solution, but choosing $\scalh \sim O(h^{-1})$ balances, with respect to $h$, the two terms on the right-hand side of \eqref{eqn: Dirichlet bilinear form}, is needed to prove  the coercivity of the bilinear form $B_{h, D}^\scal$}, and may positively affect properties of the stiffness matrix; see \S\ref{section: analysis}. 

A careful inspection shows that the problems { \eqref{eqn: PEF Dirichlet coercive problem}}, \eqref{eqn: PEF Dirichlet problem}, and \eqref{eqn: PEF Dirichlet problem0} are equivalent.

%%%%%%%%%%%%%%%%
{\textbf{\em The PE-FEM Neumann problem.}}
The Taylor series representation \eqref{eqn: Neumann representation} of the Neumann data implies that
\begin{equation}\label{eqn: Neumann condition}
	0 \approx g_N \circ \etab(\xib) -	\widetilde p\circ\etab(\xib)\;\left. \mathbf{T}^{k-1}_h\left( \nabla u_h \right)\right|_{\etab(\xib)} \cdot \mathbf n.
\end{equation}
By adding $\widetilde p(\xib) \nabla u_h \cdot \mathbf n_h$ to both sides, we can approximate the Neumann data as
\begin{equation}\label{eqn: Neumann term}
\begin{aligned}
\widetilde p(\xib) \nabla u_h \cdot \mathbf n_h &\approx  g_N \circ \etab(\xib) +  {\widetilde p(\xib)} \nabla u_h \cdot  \mathbf n_h - \widetilde p\circ\etab(\xib)\;\left.\mathbf{T}^{k-1}_h \left( \nabla u_h \right) \right|_{\etab(\xib)} \cdot \mathbf n.
\end{aligned}
\end{equation}

The discrete weak form \eqref{eqn: PEF Neumann problem0} and \eqref{eqn: Neumann bilinear form0} of \eqref{contprobn} is given by
$$
{\WTN}(u_h, v) - \big< \widetilde p {(\xib)}\,\nabla u_h \cdot \nn_h, v \big>_{\Gamma_h} = \langle\widetilde f, \,v\rangle_{\Omega_h} \quad \forall v \in V^k_h.
$$
Incorporating \eqref{eqn: Neumann term} yields the \emph{PE-FEM Neumann formulation}: seek $u_h \in V^k_h$ such that
\begin{equation}\label{eqn: PEF Neumann problem}
B_{h, N} (u_h,v) = F_{h,N}(v) \quad \forall\, v \in V^k_h,
\end{equation}
where
\begin{equation}\label{eqn: Neumann bilinear form}
B_{h, N} (u_h,v) := {\WTN} (u_h, v) + \boldsymbol{\tau}_N(u_h, v)
\end{equation}
with 
$$
\boldsymbol{\tau}_N(u_h, v) := \big< \widetilde p\circ\etab(\xib)\;  \mathbf{T}^{k-1}_h ( \nabla u_h )\big|_{\etab(\xib)}\cdot \mathbf  n -  \widetilde p{(\xib)}\,\nabla u_h \cdot \nn_h, v \big>_{\Gamma_h} 
$$
and
\begin{equation}\label{eqn: Neumann rhs term}
F_{h,N}(v) := \langle\widetilde f, \,v\rangle_{\Omega_h} + \big\langle g_N \circ \etab(\xib), v \big\rangle_{\Gamma_h} \quad \forall\, v \in V^k_h.
\end{equation}

A careful inspection shows that the problems \eqref{eqn: PEF Neumann problem} and \eqref{eqn: PEF Neumann problem0} are equivalent.

\begin{remark}
There is a price to pay for obtaining optimal convergence rates for higher-order finite element methods on polygonal domains for problems posed on non-polygonal domains, namely that the discretized systems \eqref{eqn: PEF Dirichlet problem} and \eqref{eqn: PEF Neumann problem} are not symmetric, even for given symmetric problems, i.e., even if ${\WTD}(\cdot,\cdot)$ and ${\WTN}(\cdot,\cdot)$ are symmetric bilinear forms. However, if these forms are indeed symmetric and if an iterative linear system solver is used, the additional computational cost due to any destruction of symmetry is not so burdensome. The contributions to the stiffness matrices associated with those forms, being associated with interior nodes of the mesh, are much larger than the contributions associated with the terms causing the lack of symmetry because the latter are associated with boundary nodes.
\end{remark}

%%%%%%%%%%%%%%%%%%%%%%%%%%%%%
\section{{Analysis} of the PE-FEM formulations} %\label{section: analysis}

We now show that \eqref{eqn: PEF Dirichlet coercive problem} and \eqref{eqn: PEF Neumann problem} are well posed and satisfy a \emph{polynomial preserving property}. Then, we prove that the PE-FEM formulations for $k^{th}$ order Lagrangian finite element spaces are optimally accurate in the $H^1(\Omega_h)$ norm. Additionally, we prove optimal $L^2(\Omega_h)$ convergence for the PE-FEM Dirichlet formulation under certain conditions on $\Omega_h$ and additional regularity on $u$. Throughout the section we assume that $\Omega_h$ consists of a regular mesh \cite{ciarlet2002finite} and let ${\urho}:= \min_{x \in \Omega_h} \widetilde p(x)$, ${\orho} := \max_{x\in \Omega_h} \widetilde p(x)$, and $\ukappa = \min_{x \in \Omega_h} \widetilde q(x)$.
%, and $\okappa:= \max_{x\in \Omega_h} \widetilde q(x)$.

\subsection{Well posedeness of the PE-FEM methods} \label{section: analysis}

\begin{theorem}[{\textbf{Well posedness of the PE-FEM Dirichlet approximation}}]\label{theorem: Dirichlet weak coercivity}
Let $B^\scal_{h, D}(\cdot, \cdot)$ be defined as in \eqref{eqn: Dirichlet bilinear form} with $\urho > 0$. Assume that $\scal_h \geq C_\scal h^{-1}$ with $C_\scal$ large enough and assume that $\deltah \sim {o}(h^{\frac32})$. Then, for $h$ small enough and $k=1,2,\ldots$, we have that
\begin{equation} \label{eqn: Dirichlet boundedness}
	B^{\scal}_{h, D}(u, v) \leq \Big ({\orho} + \scalh \big ( 1+ C \sum^k_{|\alphab| = 1} h^{\frac12-|\alphab|} \deltah^{|\alphab|} \big) \Big)  \| u \|_{1, \Omega_h} \| v \|_{1, \Omega_h} \quad\forall u,v \in {V^h_k}
\end{equation}
and
\begin{equation} \label{eqn: Dirichlet coercivity condition}
	B^\scal_{h,D}(u,u) \geq \gamma_D \|u\|^2_{1, \Omega_h} \quad \forall u \in V^k_h.
\end{equation}
If $g_D\circ\etab(\xib) \in H^{1/2}(\Gamma_h)$, {then \eqref{eqn: PEF Dirichlet coercive problem} has a unique solution $u_h$ and that solution satisfies the stability bound}
\begin{equation} \label{eqn: Dirichlet data control}
	\| u_h \|_{1, \Omega_h} \leq C\big( \|\widetilde f \|_{-1, \Omega_h} + \|g_D\jc{\circ \etab(\xib)}\|_{1/2, \Gamma_h}\big).
\end{equation}
\end{theorem}
The proof is provided in \S\ref{dwc}

\begin{theorem}[{\textbf{Well-posedness of the PE-FEM Neumann approximation}}]\label{theorem: Neumann well--posednessn}
Let $B_{h, N}(u,v)$,  be defined as in \eqref{eqn: Neumann bilinear form}. Assume that $\deltah \sim {o}(h^{\frac32})$, $\widetilde p , \widetilde q > 0$ in $\Omega_h$ and $\widetilde p >0$ on $\Gamma$. Then, for $h$ small enough and $k=1,2,\ldots$, we have that
\begin{equation} \label{eqn: Neumann continuity bound}
	B_{h,N}(u,v)\leq C \| u \|_{1, \Omega_h} \|v\|_{1, \Omega_h} \quad\forall u, v \in V_h^k
\end{equation}
and 
\begin{equation} \label{eqn: Neumann coercivity}
	B_{h,N}(u,u) \geq \gamma_N \|u\|^2_{\Omega_h} \quad \forall u \in V^k_h.
\end{equation}
{If $g_N\circ\etab(\xib) \in H^{-1/2}(\Gamma_h)$, then \eqref{eqn: PEF Neumann problem} has a unique solution $u_h$ and that solution satisfies the stability bound}
\begin{equation} \label{eqn: Neumann data control}
	\| u_h \|_{1, \Omega_h} \leq C {\gamma^{-1}_N}\big( \| \widetilde f \|_{-1, \Omega_h} + \|g_N\circ\etab(\xib)\|_{-1/2, \Gamma_h} \big).
\end{equation}
{Finally, for  all $u, v \in H^1(\Omega_h)$ such that $u|_{\KKK_n} \in H^{k+1}(\KKK_n)$ for all $\KKK_n \subset \Omega_h$, we have that, for $k=2,3,\ldots$,
\begin{equation} \label{eqn: Neumann continuity bound in H1}
\begin{aligned}
B_{h,N}(u,v)
	&\le C\Big[\|u\|_{1, \Omega_h} + 
		( \delta_h h^\frac32 + h^\frac52 ) \trino u \trino_{3, \Omega_h} 
		\\&
	\quad + \delta_h^\frac12 \trino u \trino_{2, \Omega_h} + \delta_h^{k-\frac12} \trino u \trino_{k+1, \Omega_h}\Big] \|v\|_{1, \Omega_h}.
\end{aligned}
\end{equation}
}
\end{theorem}
The proof is provided in \S\ref{nwc}

\begin{remark}
Theorems \ref{theorem: Dirichlet weak coercivity} and \ref{theorem: Neumann well--posednessn}, establish well posedness of the PE-FEM Dirichlet and Neumann problems for the linear diffusion-reaction equation. We remark that if a convection operator also appears along with the elliptic operator, the above analysis remains, for the most part, unchanged, i.e., it can be treated in the same manner as convection terms are handled by standard finite element methods. 
\end{remark}

%%%%%%%

\subsection{Polynomial preserving property}

The finite element space $V^k_h$ contains the {\em global} polynomial space $P_k(\Omega_h)$. Thus, a desirable property of a finite element method implemented using $V^k_h$ is for it to exactly recovery of global polynomial fields in $P_k(\Omega_h)$. Whereas this ``patch test'' is not sufficient for optimal convergence, it provides a useful diagnostic tool for code verification.

It is straightforward to show that PE-FEM preserves global polynomial fields. Given an $r \in P_k(\Omega_h)$ and a forcing function $\widehat f = L r$, where $L$ denotes, as appropriate, either the strong operator in \eqref{contprobd} for Dirichlet problems or the strong operator in \eqref{contprobn} for Neumann problems, it is clear that $r$ satisfies
\begin{equation*}
	{\WTD}(r, v) = \langle\widetilde f, v\rangle_{\Omega_h} \quad \forall v \in V^k_{h,0} \qquad \textrm{and}\qquad
	{\WTN}(r, v) = \langle\widetilde f, v\rangle_{\Omega_h} \quad \forall v \in V^k_h .
\end{equation*}
Because Taylor series preserve polynomials, the boundary conditions of \eqref{eqn: PEF Dirichlet coercive problem} and \eqref{eqn: PEF Neumann problem} are satisfied by $r$ if $g_D = r|_{\Gamma}$ and $g_N = { (\widetilde p \nabla r \cdot \nn)|_\Gamma}$. Furthermore, {if the conditions in Theorems \ref{theorem: Dirichlet weak coercivity} and \ref{theorem: Neumann well--posednessn} hold, $u_h = r$ is the unique solution of the PE-FEM equations.} Thus, we have established the following proposition. 
\begin{proposition}
	The PE-FEM methods, c.f. \eqref{eqn: PEF Dirichlet coercive problem} and \eqref{eqn: PEF Neumann problem}, are polynomial preserving.
\end{proposition}

%%%%%%%%%%%%%%%%%%%%%%%%

%%%%%%%%%%%%%%%%%
\subsection{Error estimates for PE-FEM approximations} \label{section: analysis2}

Using results from \S\ref{section: analysis}, we now prove optimal $H^1(\Omega_h)$-norm error estimates for the Dirichlet and Neumann PE-FEM problems and optimal$L^2(\Omega_h)$-norm error estimates for the Dirichlet PE-FEM problem.

%%%%%%%%%%%%%%%%%%%%%%%%

\begin{theorem}[{\textbf{$H^1(\Omega_h)$-norm error estimate for the Dirichlet PE-FEM approximation}}]\label{theorem: Dirichlet PE-FEM H1}
Assume that $\widetilde f \in H^{k-1}(\Omega)$, $g_D \in H^{k+\frac12}(\Gamma)$, and the hypotheses of Theorem  \ref{theorem: Dirichlet weak coercivity} hold with the additional assumption that $\deltah\sim o(h^\frac32)$ if $d = 2$ {and $\deltah \sim O(h^2)$ if $d = 3$}. Let $u_h \in V_h^k$ denote the solution of \eqref{eqn: PEF Dirichlet coercive problem}, $u \in H^{k+1}(\Omega)$ the solution to \eqref{eqn: weak Dirichlet}, and $\tu \in H^{k+1}(\mathbb R^\DDD)$ the extension of the latter. Then,
$$
\|\tu - u_h\|_{1,\Omega_h} \le C h^{k} (\|u\|_{k+1, \Omega} + \|f\|_{k-1,\Omega})\quad  \text{for } k=2,3,\ldots. 
$$
\end{theorem}
The proof is given in \S\ref{sec:aeeed}.

%%%%%%%%%%%%%%%%%%%%

\begin{theorem}[{\textbf{$L^2(\Omega_h)$-norm error estimate for the Dirichlet PE-FEM approximation}}]\label{theorem: Dirichlet PE-FEM L2}
Assume that $\deltah \sim O(h^2)$, {$f \in H^{k}(\Omega)$}, and that the hypotheses of Theorem  \ref{theorem: Dirichlet weak coercivity} holds. Let $u_h \in V_h^k$ denote the solution of \eqref{eqn: PEF Dirichlet coercive problem}, {$u \in H^{k+\frac32}(\Omega)$ the solution to \eqref{eqn: weak Dirichlet}, and $\tu \in H^{k+\frac32}(\mathbb R^\DDD)$ the extension of the latter}. Then,
\begin{equation} \label{eqn: DIrichlet L2 bound}
\|\tu - u_h\|_{0,\Omega_h} \le C h^{k+s} (\|u\|_{k+1, \Omega} {+ |\widetilde u|_{k+1, \Gamma_h} } + \|f\|_{k,\Omega}),\quad  \text{for } k=2,3\ldots 
\end{equation}
where $s \in (\frac12, 1]$ is a constant dependent on the largest interior angle of $\partial \Omega_h$. Given $h$ is small enough. Additionally, if $\partial \Omega_h$ is a convex polytope, we have that \eqref{eqn: DIrichlet L2 bound} holds with $s = 1$.
\end{theorem}

The proof is given in \S\ref{sec:aeeed2}.

%%%%%%%%%%%%%%%%%%%%

\begin{theorem}[{\textbf{$H^1(\Omega_h)$-norm error estimate for the Neumann PE-FEM approximation}}]\label{theorem: Neumann PE-FEM H1}
Assume that  $f \in H^{k-1}(\Omega)$,  $g_N \in H^{k-\frac12}(\Gamma_h)$, and that the hypotheses of Theorem \ref{theorem: Neumann well--posednessn} hold.  Let $u_h \in V_h^k$ denote the solution of \eqref{eqn: PEF Neumann problem}, $u \in H^{k+1}(\Omega)$ the solution to \eqref{eqn: weak Neumann}, and $\tu \in H^{k+1}(\mathbb R^\DDD)$ the extension of the latter. Then, if $\deltah \sim O(h^2)$ the we have the bound
$$
\|\tu - u_h\|_{1,\Omega_h} \le C h^{k} (\|\tu\|_{k+1} + \|f\|_{k-1,\Omega})\quad \text{for } {k=2,3,\ldots}. 
$$
\end{theorem}

The proof is given in \S\ref{sec:aeeen}.

\begin{remark}
Although our results do not include optimal $L^2$-norm error estimates for the Neumann problem, numerical results given in Section \ref{section: results} suggest that the PE-FEM formulation is optimally accurate in this case as well.
\end{remark}

%%%%%%%%%%%%%%%%%%%%%%
\section{Implementation} \label{section: implementation}

The conversion of any finite element code into a PE-FEM code is a relatively simple matter as it only requires coding the additional terms in \eqref{eqn: PEF Dirichlet problem0d} and \eqref{eqn: Neumann bilinear form0} which are relatively minor variations of the standard assembly process on each element.

%%%%%

\subsection{The mapping $\eta(\xi)$}

Whereas different choices of mappings $\etab: \Gamma_h\to\Gamma$ are possible, in our numerical experiments we use the mapping
\begin{equation}\label{eqn: nearest neighbor}
	\etab(\xib) := \underset{{\bf x}\in \Gamma}{\arg\min} |{\bf x} - \xib|
\end{equation}
that guarantees that $|\etab(\xib) - \xib| = O(h^2)$. 
In the implementation, we apply the mapping to $\xib$ belonging to the set of nodes or quadrature points lying on $\Gamma_h$. 

%%%%%

\subsection{Implementation of the PE-FEM Dirichlet problem}

In \eqref{eqn: PEF Dirichlet problem0d}, the essential boundary condition is imposed on the piecewise polynomial extensions in a weak, i.e., variational, sense. In order to compute the boundary integrals using quadrature rules, one can compute  the term $\left( E_{\KKK_{j_i}} (u_h) \circ \etab(\xib_q) \right)$ at quadrature points $\xib_q\in {\mathcal E}_i$ by evaluating the polynomial basis functions that generate $u_h|_{\KKK_{j_i}}$ at $\etab(\xib_q)$.

Alternatively, for two-dimensional Type A meshes, it is possible to prescribe the Dirichlet condition in a strong sense
$$
E_{\KKK_{j_i}} (u_h) \circ \etab(\xib_i) = g_D\circ\etab(\xib_i),
$$
for all nodes $\xib_i \in \Gamma_h$ associated with the degrees of freedom of $W_h^k$. Here, $\KKK_{j_i}$ denotes the element whose closure contains $\xib_i$. The element $\KKK_{j_i}$ is not uniquely defined if $\xib_i$ is a vertex of the triangulation. However, in this case, for two-dimensional Type A meshes we have that $\xib_i \in \Gamma \cap \Gamma_h$ and therefore the extension operators reduce to the identity operator.

%%%%%%%%%%%%

\subsection{Implementation of the PE-FEM Neumann problem}

The PE-FEM Neumann problem can be obtained by adding the term \eqref{eqn: Neumann bilinear form0}
\begin{equation*}
	\sum_i \Big\langle  \widetilde p\circ\etab(\xib)\;   \nabla \left( E_{\KKK_{j_i}} (u_h) \circ \etab(\xib) \right) \cdot \mathbf  n -  \widetilde p\,\nabla u_h \cdot \nn_h, v \Big\rangle_{\EEE_i} 
\end{equation*}
to a standard finite element implementation, and by evaluating the Neumann data $g_N$ at the points on the true boundary $\Gamma$ using the mapping $\etab$.
In particular, when computing the integrals using quadrature rules, one can compute  the term $\nabla \left( E_{\KKK_{j_i}} (u_h) \circ \etab(\xib_q) \right)$ at quadrature points by evaluating the gradient of the polynomial basis functions that generates $u_h|_{\KKK_{j_i}}$ at $\etab(\xib_q)$ and can compute the right hand side by evaluating $g_N$ at $\etab(\xib_q)$. The outer unit normal vectors  $\nn_h$ at $\xib_q$ and $\nn$ at $\etab(\xib_q)$ need to be computed as well.

%%%%%%%%%%%%%%%%%%%%%%%%%

\section{Numerical Examples} \label{section: results}
In this section, we present illustrative numerical results for the Dirichlet and Neumann PE-FEM methods for both a convex and a non-convex domain. 

\subsection{Convex domain}
The convex domain $\Omega$ considered is the unit circle centered at $(0, 0)$ having radius $1$. The coefficients are given by $p(x)=q(x)=1$ and the right-hand side $f$ is determined so that the exact solution is given by $u = \cos(x) \cos(y)$. A sample PE-FEM approximate solution is plotted in Figure \ref{solution} (left). We report on the convergence history for the PE-FEM method in Table \ref{table 1}. We observe that optimal $H^1(\Omega_h)$ and $L^2(\Omega_h)$ convergence rates are achieved in all cases.

\begin{figure}[h!]
\centerline{
	\includegraphics[width = 0.35\textwidth]{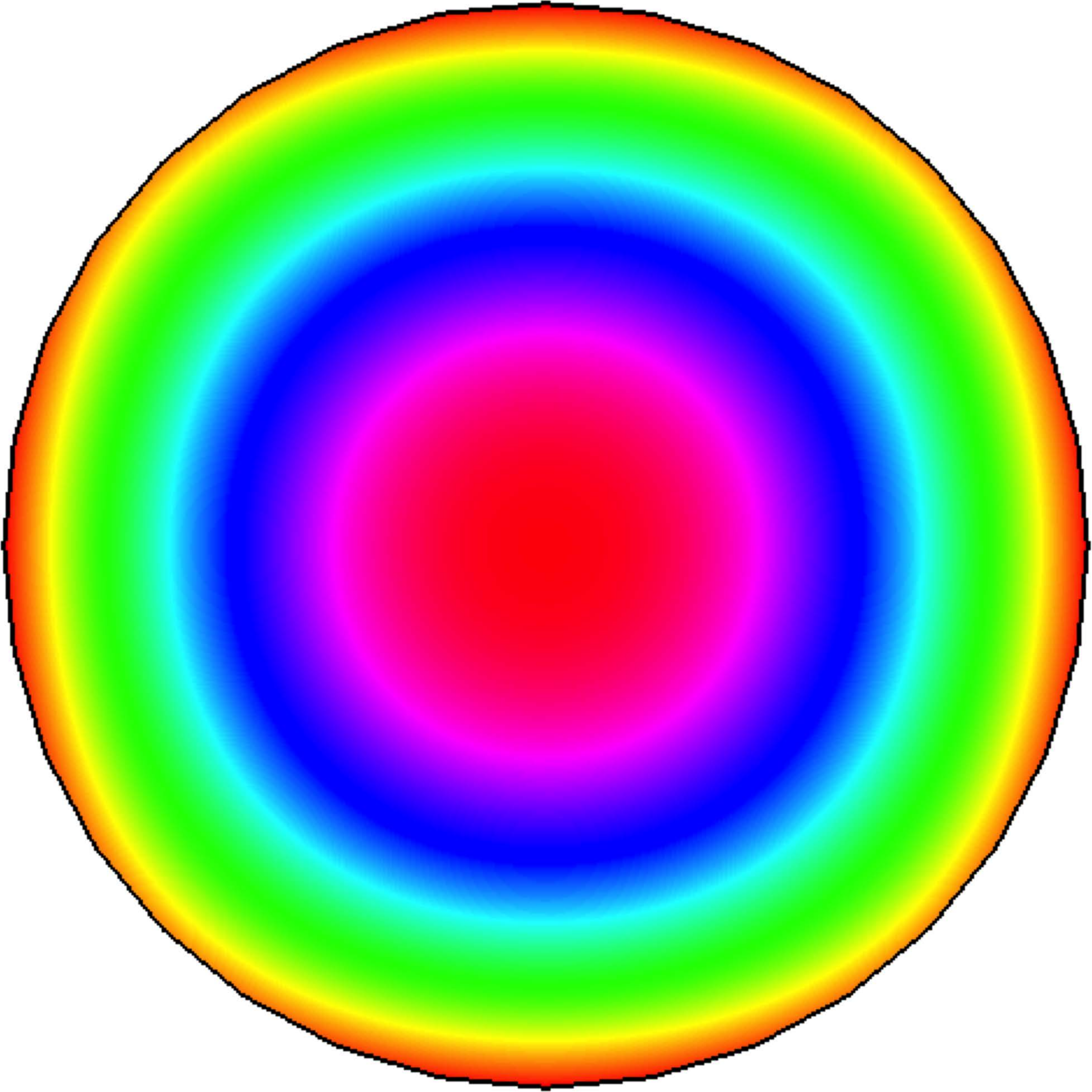}\qquad
	\includegraphics[width = 0.35\textwidth]{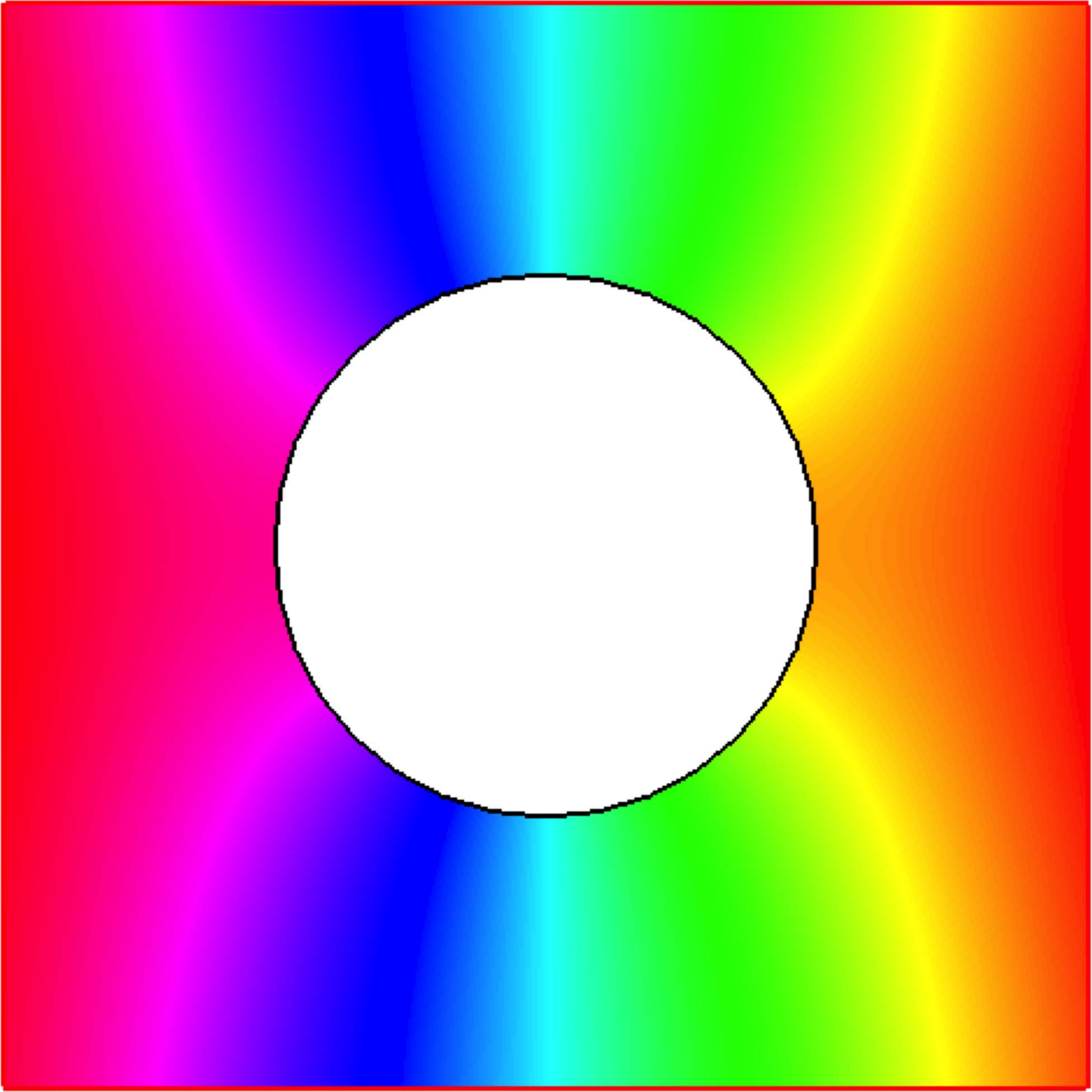}
	}
	\caption{Numerical solutions computed by the PE-FEM method. Left: solution of the circular domain problem computed using the Neumann PE-FEM method with cubic elements. Right: solution for the nonconvex domain computed using the Dirichlet PE-FEM method with cubic elements.}\label{solution}
\end{figure}

\begin{table}
\caption{Convergence histories for the PE-FEM Dirichlet and Neumann PE-FEM approximations for the convex domain example. Optimal convergence rates are achieved with respect to both the $L^2(\Omega_h)$ and $H^1(\Omega_h)$ norms.}\label{table 1} 
\begin{center}
\begin{tabular}{|c||c|c||c|c|}
  \hline
  \multicolumn{5}{|c|}{Quadratic elements ($k=2$)}\\
  \hline
 & \multicolumn{2}{|c||}{Dirichlet} & \multicolumn{2}{|c|}{Neumann} \\
  \cline{2-5}
   $h$  & $\| u - u_h \|_{0, \Omega_h}$   & $\| u - u_h \|_{1, \Omega_h}$ 
    & $\| u - u_h \|_{0, \Omega_h}$   & $\| u - u_h \|_{1, \Omega_h}$ \\
   \hline
0.583095 & 6.83996e-04 & 1.20924e-02 & 8.71677e-03 & 1.18623e-02 \\
0.315543 & 8.71107e-05 & 2.96301e-03 & 1.29589e-03 & 2.93169e-03 \\
0.165152 & 1.07759e-05 & 7.21901e-04 & 1.55799e-04 & 7.11529e-04 \\
0.080322 & 1.28123e-06 & 1.81096e-04 & 1.99808e-05 & 1.79352e-04 \\
0.045221 & 1.59731e-07 & 4.41786e-05 & 2.22425e-06 & 4.39835e-05 \\
   \hline
   Rate &3.2283     &  2.1605  & 3.1932 & 2.1558\\
   \hline
\end{tabular}
\vskip2pt
\begin{tabular}{|c||c|c||c|c|}
  \hline
  \multicolumn{5}{|c|}{Cubic elements ($k=3$)}\\
  \hline
   & \multicolumn{2}{|c||}{Dirichlet} & \multicolumn{2}{|c|}{Neumann} \\
 \cline{2-5}
   $h$ & $\| u - u_h \|_{0, \Omega_h}$   & $\| u - u_h \|_{1, \Omega_h}$ 
    & $\| u - u_h \|_{0, \Omega_h}$   & $\| u - u_h \|_{1, \Omega_h}$ \\
   \hline
0.583095 & 3.11001e-05 & 7.19118e-04 & 2.74664e-04 & 7.11022e-04 \\ 
0.315543 & 1.67332e-06 & 7.63497e-05 & 1.96586e-05 & 7.61091e-05 \\
0.165152 & 1.06597e-07 & 1.00175e-05 & 9.82364e-07 & 9.95999e-06 \\
0.080322 & 6.87903e-09 & 1.34472e-06 & 6.63858e-08 & 1.34119e-06 \\
0.045221 & 4.33984e-10 & 1.66115e-07 & 2.73366e-09 & 1.65852e-07 \\
   \hline
   Rate & 4.2922     & 3.2024 & 4.4254 &  3.1993\\
   \hline
\end{tabular}
\vskip2pt
\begin{tabular}{|c||c|c||c|c|}
  \hline
  \multicolumn{5}{|c|}{Quartic elements ($k=4$)}\\
    \hline
   & \multicolumn{2}{|c||}{Dirichlet} & \multicolumn{2}{|c|}{Neumann} \\
\cline{2-5}
   $h$                   & $\| u - u_h \|_{0, \Omega_h}$   & $\| u - u_h \|_{1, \Omega_h}$ 
   & $\| u - u_h \|_{0, \Omega_h}$   & $\| u - u_h \|_{1, \Omega_h}$ \\
   \hline
0.583095 & 6.02698e-07 & 2.36345e-05 & 3.0114e-05 & 2.57857e-05 \\
0.315543 & 2.24273e-08 & 1.41265e-06 & 1.72058e-06 & 1.59270e-06 \\
0.165152 & 6.36060e-10 & 8.36391e-08 & 4.75739e-08 & 8.54912e-08 \\
0.080322 & 1.73323e-11 & 5.12435e-09 & 1.41032e-09 & 5.11654e-09 \\
   \hline
   Rate & 5.2938     & 4.2606 & 5.0798 & 4.3162\\
   \hline
\end{tabular}
\end{center}
\end{table}

\subsection{Nonconvex domain}

The non-convex domain $\Omega$ considered is the region within the unit square $[-0.5, 0.5]^2$ and outside of the circle of radius $\frac{1}{4}$ centered at (0,0). The coefficients are given by $p(x)=q(x)=1$ and the right-hand side $f$ is determined so that the exact solution is given by  $u = -\frac{17}{16}\frac{x}{x^2 + y^2}$. We use the given Dirichlet boundary conditions on the outer boundary of the square because no special treatment is required for applying boundary conditions along straight edges. The PE-FEM conditions are utilized on the interior circular boundary. A sample PE-FEM approximate solution is plotted in Figure \ref{solution} (right).  The convergence history for this numerical experiment is reported on in Table \ref{table 3}. We again observe that optimal $H^1(\Omega_h)$ and $L^2(\Omega_h)$ convergence rates are achieved in all cases.

\begin{table}
\caption{Convergence histories for PE-FEM Dirichlet and Neumann PE-FEM approximations for the non-convex domain example. Optimal convergence rates are achieved with respect to both the $L^2(\Omega_h)$ and $H^1(\Omega_h)$ norms.}\label{table 3}
\begin{center}
\begin{tabular}{|c||c|c||c|c|}
  \hline
  \multicolumn{5}{|c|}{Quadratic elements ($k=2$)}\\
  \hline
 & \multicolumn{2}{|c||}{Dirichlet} & \multicolumn{2}{|c|}{Neumann} \\
  \cline{2-5}
   $h$  & $\| u - u_h \|_{0, \Omega_h}$   & $\| u - u_h \|_{1, \Omega_h}$ 
    & $\| u - u_h \|_{0, \Omega_h}$   & $\| u - u_h \|_{1, \Omega_h}$ \\
   \hline
0.353553 & 6.39714e-03 & 1.80907e-01 & 1.57225e-02 & 2.15036e-01  \\
0.241519 & 1.09194e-03 & 5.26226e-02 & 2.95796e-03 & 5.63867e-02 \\ 
0.136382 & 1.53631e-04 & 1.4597e-02 & 4.05368e-04 & 1.44241e-02  \\
0.0714398 & 2.23882e-05 & 3.99222e-03 & 5.7218e-05 & 3.91765e-03 \\ 
0.0383488 & 2.88218e-06 & 1.00481e-03 & 8.15548e-06 & 9.91063e-04 \\ 
   \hline
   Rate &  3.3877    & 2.2749  & 3.3502  & 2.3520\\
   \hline
\end{tabular}
\vskip2pt
\begin{tabular}{|c||c|c||c|c|}
  \hline
  \multicolumn{5}{|c|}{Cubic elements ($k=3$)}\\
  \hline
 & \multicolumn{2}{|c||}{Dirichlet} & \multicolumn{2}{|c|}{Neumann} \\
  \cline{2-5}
   $h$  & $\| u - u_h \|_{0, \Omega_h}$   & $\| u - u_h \|_{1, \Omega_h}$ 
    & $\| u - u_h \|_{0, \Omega_h}$   & $\| u - u_h \|_{1, \Omega_h}$ \\
   \hline
0.353553 & 2.22751e-03 & 8.49656e-02 & 9.25961e-03 & 1.23404e-01 \\
0.241519 & 1.77764e-04 & 1.31537e-02 & 1.31772e-03 & 1.73723e-02 \\
0.136382 & 1.41219e-05 & 1.96634e-03 & 1.15194e-04 & 2.17373e-03 \\ 
0.0714398 & 1.14619e-06 & 2.87068e-04 & 7.63884e-06 & 2.90969e-04 \\
0.0383488 & 7.37505e-08 & 3.61339e-05 & 6.44114e-07 & 3.61279e-05 \\
   \hline
   Rate & 4.5017     & 3.3946  & 4.2777 & 3.5698\\
   \hline
\end{tabular}
\vskip2pt
\begin{tabular}{|c||c|c||c|c|}
  \hline
  \multicolumn{5}{|c|}{Quartic elements ($k=4$)}\\
  \hline
 & \multicolumn{2}{|c||}{Dirichlet} & \multicolumn{2}{|c|}{Neumann} \\
  \cline{2-5}
   $h$  & $\| u - u_h \|_{0, \Omega_h}$   & $\| u - u_h \|_{1, \Omega_h}$ 
    & $\| u - u_h \|_{0, \Omega_h}$   & $\| u - u_h \|_{1, \Omega_h}$ \\
   \hline
0.353553 & 7.14676e-04 & 4.08694e-02 & 2.03226e-03 & 5.66714e-02 \\ 
0.241519 & 3.26085e-05 & 3.23267e-03 & 2.85442e-04 & 4.66504e-03 \\ 
0.136382 & 1.77419e-06 & 2.59726e-04 & 1.71681e-05 & 3.17316e-04 \\
0.0714398 & 7.69598e-08 & 2.01136e-05 &  4.39592e-07 & 2.11306e-05 \\
0.0383488 & 2.43552e-09 & 1.27258e-06 & 2.27363e-08 & 1.29139e-06 \\
   \hline
   Rate & 5.4794   & 4.5304 &  5.1745 & 4.6992\\
   \hline
\end{tabular}
\end{center}
\end{table}

%%%%%%%%%%%%%%%%%%%%%%%%
\section{Concluding remarks}\label{sec:conclusion}

In this paper, we have proposed a numerical method to determine optimally accurate finite element approximation of solutions of second-order elliptic boundary value problems based on polygonal approximations $\Omega_h$ of domains $\Omega$ having smooth curved boundaries. Polynomial extensions from the boundary of $\Omega_h$ to the boundary of $\Omega$ are instrumental in achieving optimal convergence rates. For Dirichlet boundary conditions, the stability and optimal convergence with respect to the $H^1(\Omega_h)$ and $L^2(\Omega_h)$ norms is proved whereas for Neumann boundary conditions, optimal convergence with respect to the $H^1(\Omega_h)$ norm is proved. Numerical tests are used to illustrate the theory as well as to show that optimal convergence rates are also achieved for errors measured in the $L^2(\Omega_h)$ norm, even for the case of Neumann boundary conditions.
 
In the future, we will explore applying this method to other equations and engineering benchmark problems. We will also consider using this approach as a mechanism for achieving higher-order accuracy for solutions of interface problems.

\nocite{*}
%\bibliography{bibliography}

%
%				APPENDIX
%

\appendix

%%%%%%%%%%%%%%%%%%%%%%%

\section{Analysis of Taylor polynomials on boundaries}

In this section, we provide technical lemmas pertaining to properties of the Taylor series extensions we have defined in \S\ref{sec: Taylor approximation}. 

\begin{lemma}\label{lm: Taylor bound}
Let $v\in L^2(\Omega_h)$. Then,
$$
\|  T^k_h (v) |_{\etab(\xib)} \|_{0, \Gamma_h} \le C \deltah^{-\frac12}  \|  v  \|_{0,\Omega_h}.
$$
\end{lemma}

\begin{proof}
We have that $T^k_h (v)$ is a polynomial on each element domain $S^{i,\mmm}$ so that the $L^{\infty}$ norm on $\etab(\EEE_i)$ is bounded by the $L^{\infty}$ norm in $S^{i,\mmm}$. It then follows that
\begin{equation*}
\begin{aligned}
 \|   T^k_h (v)|_{\etab(\xib)}  \|_{0, \EEE_i\cap S^{i,\mmm}}
&\le |\EEE_i \cap S^{i,\mmm}|^\frac12  \|   T^k_h (v) |_{\etab(\xib)}  \|_{L^\infty  (\EEE_i\cap S^{i,\mmm})} \\
&\le C \deltah^{\frac{d-1}{2}}  \|   T^k_h (v) |_{\etab}  \|_{L^\infty  (\etab (\EEE_i)\cap S^{i,\mmm})} 
\le C \deltah^{\frac{d-1}{2}}  \|  T^k_h (v)  \|_{L^\infty(S^{i,\mmm})} 
\end{aligned}
\end{equation*}
because $\textrm{diam}(S^{i,\mmm}) \leq C\delta_h$.
From \cite[Corollary 4.1.15]{brenner2007mathematical} and after using a scaling argument and noticing that diam$(\sigma^{i,\mmm}) \le C \deltah$, we have that $ \|  T^k_h (v)  \|_{L^\infty(S^{i,\mmm})} \le C\deltah^{-d} \| v\|_{L^1(\sigma^{i,\mmm})}$. It then follows that
\begin{equation*}
\begin{aligned}
 \|   T^k_h (v)|_{\etab(\xib)}  \|_{0, \EEE_i\cap S^{i,\mmm}} &\le C \deltah^{-\frac{d+1}{2}}  \|  v  \|_{L^1(\sigma^{i,\mmm})} \\
&\le C \deltah^{-\frac{d+1}{2}} |\sigma^{i,\mmm}|^{\frac12} \|  v \|_{0,\sigma^{i,\mmm}} 
\le C \deltah^{-\frac12} \|  v \|_{0,\sigma^{i,\mmm}} .
\end{aligned}
\end{equation*}
We conclude the proof by summing $ \|   T^k_h (v)|_{\etab(\xib)}  \|_{0, (\EEE_i\cap S^{i,\mmm})}^2$ over all $i, \mmm$ and noticing that the $\sigma^{i,\mmm}$ are pairwise disjoint.
\end{proof}

\begin{lemma} \label{lm: Taylor approximation}
Let $U \subset \mathbb R^\DDD$ be any domain such that $\bigcup_{i,j}S^{i,\mmm} \subset U $ and  $v \in H^{k+1}(U)$ and let $m \in \mathbb N_0$. If $m+1 \leq k$, then
$$
{\trino}   T^k_h (v) |_{\etab(\xib)} -v\circ\etab(\xib) {\trino}_{m, \Gamma_h} \le C_{\Omega_h} \deltah^{k - m+\frac12} |v|_{k+1, U}.
$$
\end{lemma} 
\begin{proof}
Recall the Sobolev inequality $W^{s}_p(U) \hookrightarrow C(\overline U)$ if $ s - \frac{d}{p} > 0$. Because $k\ge m+1$ and $d = 2, 3$, we have that $v \in  W^{m}_\infty(U)$ and $D^{\alpha} v \in C(\overline U)$ for $|\alpha| \le m$. For $v \in W^{k,\infty}(\Omega)$, H\"older inequality implies that
$$
\| v \|_{m, \Omega} \le \sqrt{m+1}\, |\Omega|^{\frac12} \|v\|_{W^{m,\infty}(\Omega)}
$$
Therefore, using techniques similar to those used in the proof of Lemma \ref{lm: Taylor bound}, we have that
\begin{equation*}
\begin{aligned}
 \|   T^k_h (v) |_{\etab(\xib)} -v\circ\etab(\xib)  \|_{m, \EEE_i\cap S^{i,\mmm}} 
& \le \sqrt{m+1} \, |\EEE_i\cap S^{i,\mmm}|^{\frac12}  \|T^k_h (v) - v \|_{W^m_\infty( \etab(\EEE_i)\cap S^{i,\mmm})} \\
& \le C_{\Omega_h,m}  \deltah^{\frac{d-1}{2}}  \|T^k_h (v) - v \|_{W^m_\infty( S^{i,\mmm})} \\
& \le C_{\Omega_h,m} \, \deltah^{k-m+\frac12}  |v|_{k+1, S^{i,\mmm}},
\end{aligned}
\end{equation*}
where the last inequality follows from \cite[Proposition 4.3.2]{brenner2007mathematical}. We complete the proof by summing the squares of this inequality over $i,\mmm$ and noticing that $S^{i,\mmm}$ are pairwise disjoint.
\end{proof}

%\begin{corollary}\label{lm: Taylor approximation corollary}
%If the assumptions of Lemma \ref{lm: Taylor approximation} are satisfied, then, for $k=2,3,\ldots$, the bound 
%$$
%\|  T^k_h (v) |_{\etab(\xib)} -v\circ\etab(\xib) \|_{1/2, \Gamma_h} \le C_{\Omega_h} \deltah^k |v|_{k+1,U},
%$$
%holds with $U$ defined as in Lemma \ref{lm: Taylor approximation}.
%\end{corollary}
%\begin{proof}
%	This inequality is derived simply by utilizing the norm interpolation inequality \cite[Theorem 5.2]{adams2003sobolev}, i.e., 
%	\begin{equation}
%		\| w \|_{1/2, \Gamma_h} \leq C\beta^{-1}\|w\|_{0, \Gamma_h} + \beta \|w\|_{1, \Gamma_h}
%	\end{equation}
%	with $w = T^k_h (v) |_{\etab(\xib)} -v\circ\etab(\xib) $ and $\beta = \delta^\frac12$.
%\end{proof}

\begin{remark}
If $U = \mathbb R^\DDD$ and $v \in H^{k+1}(\mathbb R^\DDD)$ is an extension of a function $w \in H^{k+1}(\Omega)$, then, the seminorm $|v|_{k+1, U}$ in the upper bounds of the inequalities of Lemma \ref{lm: Taylor approximation} and can be replaced by $\|w\|_{k+1, \Omega}$ by virtue of the continuity of the extension operator. 
\end{remark}

\begin{lemma} \label{lm: Taylor approximation 2}
Let $v \in L^2(\Omega_h)$ such that  $v|_{\KKK_n} \in H^{k+1}(\KKK_n) \cap C^0(\KKK_n)$ for every $\KKK_n \subset \Omega_h$. Then,
$$
 \|  T^k_h (v) |_{\etab(\xib)} - v(\xib)  \|_{0, \Gamma_h} \le C_{\Omega_h}  (\deltah^{\frac12} \trino v\trino _{1,\Omega_h} + \deltah^{k+\frac12} \trino v\trino _{k+1,\Omega_h}).
$$
Note that, in contrast with Lemma \ref{lm: Taylor approximation}, $v$ is evaluated at $\xib$ whereas $T_h^k$ is evaluated at $\etab(\xib)$.
\end{lemma} 
\begin{proof}
Using techniques similar to those used in the proof of Lemma \ref{lm: Taylor bound}, we obtain
\begin{equation*} \label{tmp Taylor approx 2}
 \|  T^k_h (v) |_{\etab(\xib)} - v(\xib)  \|_{0, \Gamma_h} \le
  \|  T^k_h (v) |_{\xib} - v(\xib)  \|_{0, \Gamma_h}  +  \|  T^k_h (v) |_{\etab(\xib)} -  T^k_h (v) |_{\xib}  \|_{0, \Gamma_h}. 
\end{equation*}
To bound the first term on the right-hand side of \eqref{tmp Taylor approx 2} we proceed as in the proof of  Lemma \ref{lm: Taylor approximation}, but now working on the domains $S^{i,\mmm} \cap \KKK_{j_i}$ (which are still star-shaped with respect to $\sigma^{i,\mmm}$) instead of $S^{i,\mmm}$. Also, because $v \notin H^{k+1}(\Omega_h)$, we must instead bound the error by a broken norm, e.g., we have that
$$
 \|  T^k_h (v) |_{\xib} - v(\xib)  \|_{0, \Gamma_h} \le C \deltah^{k+\frac12} \trino v\trino _{k+1, \Omega_h}.
$$
On each $S^{i,\mmm}$, the second term on the right-hand-side of \eqref{tmp Taylor approx 2} features the difference of the same polynomial evaluated at the different points $\etab(\xib)$ and $\xib$. Hence, using the classical first-order Taylor expansion on the polynomial, we have, for $\xib, \etab(\xib) \in S^{i,\mmm}$,
$$
\begin{aligned}
 \|  T^k_h (v) |_{\etab(\xib)} -  T^k_h (v) |_{\xib}  \|_{\EEE_i \cap S^{i,\mmm}} 
&=  \| \nabla   T^{k}_h (v) |_{\xib'} (\etab(\xib)-\xib)  \|_{\EEE_i \cap S^{i,\mmm}}\\
&\le C \deltah  \|   {\bf T}^{k-1}_h (\nabla v) |_{\xib'}  \|_{\EEE_i \cap S^{i,\mmm}} 
\le C \deltah^{\frac12} \, \| \nabla v \|_{0, \sigma^{i,\mmm}},
\end{aligned}
$$
where $\xib' = t \xib + (1-t)\etab(\xib)$ for some $t \in [0,1]$. For the first inequality we used the fact that, on each $S^{i,\mmm}$, $D^{\alpha} T^k_h (v) = T^{k-|\alpha|}_h (D^{\alpha} v)$; see \cite[Proposition 4.1.17]{brenner2007mathematical}. For the last inequality we proceeded as in the proof of Lemma \ref{lm: Taylor bound}.
\end{proof}

\begin{lemma} \label{lm: Taylor discrete bound}
If $v \in \overline V_h^{k}$, then
\begin{equation}\label{eqn: Taylor discrete bound 1}
\Big\| T^{k',k}_{h}(v)\Big|_{\etab(\xib)}\big\|_{0, \Gamma_h} \le C_{\Omega,k}\Big (\sum^k_{|\alphab| = k'} h^{-|\alphab|-\frac12} \deltah^{|\alphab|} \Big) \|v\|_{0,\Omega_h}.
\end{equation}
In addition, if $v \in V_h^{k}$ and $m >0$, then
\begin{equation}\label{eqn: Taylor discrete bound 2}
\Big\| T^{k',k}_{h}(v)\big|_{\etab(\xib)}\Big\|_{0, \Gamma_h} \le C_{\Omega,k}\Big(\sum^k_{|\alphab| = k'} h^{-|\alphab| + \frac12} \deltah^{|\alphab|} \Big) |v|_{1,\Omega_h}.
\end{equation}
\end{lemma}

\begin{proof}
Let $v \in \overline V_h^{k}$, then it follows from the definition of $T^{k',k}_h(\cdot)$ that
$$
\begin{aligned}
\Big\|T^{k',k}_{h}&(v)\big|_{\etab(\xib)}\Big\|_{0, \EEE_i} = \Big\|\sum^k_{|\alphab| = k' } \frac{1}{\alphab !} D_h^\alphab v(\xib)|\xib - \etab(\xib)|^\alphab\Big\|_{0,  \EEE_i} \\
& \le \sum^k_{|\alphab| = k'} \deltah^{|\alphab|} \frac{1}{\alphab !} \|D_h^\alphab v(\xib)\|_{0, \EEE_i}  
\le C h^{\frac{d-1}{2}} \sum^k_{|\alphab| = k'} \deltah^{|\alphab|} \frac{1}{\alphab !} \| D_h^\alphab v\|_{L^\infty(\EEE_i)} \\
& \le C h^{\frac{d-1}{2}}  \sum^k_{|\alphab| = k'} \deltah^{|\alphab|} \frac{1}{\alphab !} \| D_h^\alphab v\|_{L^\infty(\KKK_{j_i})} 
 \le C h^{-\frac{1}{2}} \sum^k_{|\alphab| = k'} \deltah^{|\alphab|} \frac{1}{\alphab !} \|D_h^\alphab v\|_{0, \KKK_{j_i}}.
\end{aligned}
$$
Using the inverse inequality we have that
$$
\Big\| T^{k',k}_{h}(v)\big|_{\etab(\xib)}\Big\|_{0, \EEE_i} \le C h^{-\frac{1}{2}} \sum^k_{|\alphab| = 1} h^{-|\alphab|} \deltah^{|\alphab|} \frac{1}{\alphab !} \|v\|_{0, \KKK_{j_i}}
$$
so that \eqref{eqn: Taylor discrete bound 1} follows by summing the terms $\|T^{k}_{h}(v)|_{\etab(\xib)} - v(\xib)\|_{0, \EEE_i}^2$.

For $m > 0$, we can write
$$
\begin{aligned}
\Big\| T^{k',k}_{h}(v)\big|_{\etab(\xib)}\Big\|_{0, \EEE_i} &\le C h^{-\frac{1}{2}} \sum^k_{|\alphab| = k'} \deltah^{|\alphab|} \frac{1}{\alphab !} \|D_h^\alphab v\|_{0, \KKK_{j_i}} \\
& \le C h^{-\frac{1}{2}} \sum^k_{|\alphab| = k'} h^{1-|\alphab|} \deltah^{|\alphab|} \frac{1}{\alphab !} |v|_{1, \KKK_{j_i}},
\end{aligned}
$$
where for the last inequality we used the fact that $D^\alphab v = D^{\beta_{\alphab}} (\partial_{x_{i_\alphab}} v)$ for some $\beta_{\alphab}$ such that $|\beta_{\alphab}| = |\alphab| - 1$ and some $i_\alphab \in \{1,\ldots,d\}$. Hence, 
$$
\|D^\alphab v\|_{0, \Omega_h} = \|D^{\beta_{\alphab}} (\partial_{x_{i_\alphab}} v)\|_{0, \Omega_h} \le C h^{- |\beta_\alphab|} \| \partial_{x_{i_\alphab}} v\|_{0, \Omega_h} \le C h^{1 - |\alphab|} |v |_{1, \Omega_h}
$$
so that \eqref{eqn: Taylor discrete bound 2} follows by summing the square of these terms.
\end{proof}

\begin{lemma} \label{lm: u approx}
Let $v \in L^2(\Omega_h)$ and $v|_{\KKK_n} \in H^{2}(\KKK_n)$, for every $\KKK_n \subset \Omega_h$. Then,
$$
 \| v \|_{0, \Gamma_h} \le C_{\Omega_h}  (h^{-\frac12} \|v\|_{0,\Omega_h} + h^{\frac{3}{2}} \trino v\trino _{2,\Omega_h}).
$$
\end{lemma}
 
\begin{proof}
Letting $Q_i^1(u)$ denote the averaged Taylor polynomial of degree $1$ defined on the maximal ball included in $\KKK_{j_i}$, we have
$$
\begin{aligned}
\|v\|_{0,\EEE_i} &\le C  |\EEE_i|^\frac12 \| v\|_{L^\infty(\EEE_i)}
                              \le C h^{\frac{d-1}{2}} \| v\|_{L^\infty(\KKK_{j_i})} \\
                              &\le C h^{\frac{d-1}{2}} (\| v - Q_i^1 v\|_{L^\infty(\KKK_{j_i})} + \| Q_i^1 v\|_{L^\infty(\KKK_{j_i})} )\\
                              &\le C h^{\frac{d-1}{2}} (h^{2-\frac{d}{2}}  \| v \|_{2,\KKK_{j_i}} + h^{-d} \| v\|_{L^1(\KKK_{j_i})} )
                              \le C h^{\frac{3}{2}}  \| v \|_{2,\KKK_{j_i}} + C h^{-\frac12} \| v\|_{0, \KKK_{j_i}} ,
\end{aligned}
$$
where, for the fourth inequality we have used \cite[Corollary 4.1.13]{brenner2007mathematical} and \cite[Proposition 4.3.2]{brenner2007mathematical} on the domain $\KKK_{j_i}$ which is star shaped with respect to $\sigma^{i,\mmm}$.
\end{proof}

\section{Analysis of extension error}
In this section, we present results pertaining to the error between different extensions of a function in $H^{k}(\Omega)$ into $H^k(\Omega_h)$.

\begin{lemma} \label{lem:ext}
Let $f \in H^k(\Omega)$ and let $\widetilde f$ and $\overline f$ denote two extensions of $f$ in $H^{k}(\mathbb R^\DDD)$ such that $\|\widetilde f\|_{k, \mathbb R^\DDD} \le C \| f \|_{k, \Omega}$ and $\|\overline f\|_{k, \mathbb R^\DDD} \le C \| f \|_{k, \Omega}$. Then, for $m,k \in \mathbb Z$ such that $0 \le m \le k$,
$$
\|\widetilde f - \overline f\|_{m, \Omega_h} {=} C_{k,d} \deltah^{k-m} \| f \|_{k, \Omega}.
$$ 
\end{lemma}
\begin{proof}
Because $\widetilde f = \overline f = f$ in $\Omega$, 
$$
\|\widetilde f - \overline f\|_{k, \Omega_h} \le \|\widetilde f - \overline f\|_{k, \Omega_h \setminus (\Omega \cap \Omega_h)}.
$$
Using the Bramble-Hilbert lemma (see \cite[Lemma 4.3.8]{brenner2007mathematical}), we have 
$$
\begin{aligned}
\|\widetilde f - T^{k-1}_h f\|_{m,  \Omega_h \setminus (\Omega \cap \Omega_h)}^2 &\le \sum_{i,j} \|\widetilde f - T^{k-1}_h f\|_{m, S^{i,\mmm}}^2 
\le  C_{k,d} \sum_{i,j} \deltah^{2(k-m)} |\widetilde f|_{k, S^{i,\mmm}}^2 
\\&\le  C_{k,d} \deltah^{2(k-m)} |\widetilde f|_{k, \mathbb R^\DDD}^2 
\le  C_{k,d} \deltah^{2(k-m)} \| f \|_{k, \Omega}^2.
\end{aligned}
$$ 
Therefore
$$
\|\widetilde f - T^{k-1}_h f\|_{m, \Omega_h} \le C_{k,d} \deltah^{k-m} \| f \|_{k, \Omega_h},
$$
which holds for $\overline f$ in place of $\widetilde f$ as well.
The lemma follows from noticing that
$$
\|\widetilde f - \overline f\|_{m, \Omega_h} \le \|\widetilde f - T^{k-1}_h f\|_{m, \Omega_h} + \|\widehat f - T^{k-1}_h f\|_{m, \Omega_h}.
$$
\end{proof}

\begin{lemma} \label{lem:ext error}
Let $f$, $\widetilde f$, and $\overline f $ be as in Lemma \ref{lem:ext}. Then,
$$
\big|\langle\widetilde f - \overline f, v \rangle_{\Omega_h}\big| \le C_{k,d} \deltah^{k +  \frac{1}{2} - \frac{1}{q}}  \|f\|_{k, \Omega} \|v\|_{1, \Omega_h} \quad \forall v \in H^1(\Omega_h),
$$
for $0 \le m \le k$ and  $2\le q < 6$ if $\Omega_h \subset \mathbb R^3$ and $2\le q < \infty$ if $\Omega_h \subset \mathbb R^2$. 
\end{lemma}
\begin{proof}
Let $\Omega_h^{\text{diff}} := \Omega_h \setminus (\Omega \cap \Omega_h)$ and note that $|\Omega_h^{\text{diff}} | \sim O(\deltah)$. Because $H^1(\Omega_h)$ is compactly embedded in $L^q(\Omega_h)$, we have, for $p = ( \frac{1}{2} - \frac{1}{q}  )^{-1}$,
$$
\begin{array}{ll}
|\langle\widetilde f - \overline f, v \rangle_{\Omega_h}| &  = \left|\left<\widetilde f - \overline f, v \right>_{\Omega_h}\right|_{\Omega_h^{\text{diff}}} 
\le \| \widetilde f - \overline f \|_{0, \Omega_h^{\text{diff}}}\,  \|v \|_{L^q(\Omega_h^{\text{diff}})}\, \| 1 \|_{L^p(\Omega_h^{\text{diff}})} \\
& \le C_{k,d} \, \deltah^{k}  \|f\|_{k, \Omega} \, \|v \|_{L^q(\Omega_h)} \, |\Omega_h^{\text{diff}}|^{\frac{1}{p}} 
\le C_{k,d} \, \deltah^{k + \frac{1}{p}} \, \|f\|_{k, \Omega} \,\|v\|_{1, \Omega_h},
\end{array}
$$
where for the second inequality we used Lemma \ref{lem:ext}.
\end{proof}

%%%%%%%%%%%%%%%%%%%%%%%%%%%%%

\section{Continuity of the $L^2$ projection in $H^{\frac12}$}

In this section, we analyze the stability of the $L^2$ projection operator in the $H^{1/2}(\Gamma_h) $ topology.

\begin{lemma} \label{continuity of projection}
	Let $\pi_h: L^2(\Gamma_h) \rightarrow W^k_h$ be the $L^2(\Gamma_h)$ projection from $L^2(\Gamma_h)$ into the discrete trace approximation space. Then, 
for all $w \in H^{1/2}(\Gamma_h)$, we have that
\begin{equation*}
	\|\pi_h w \|_{1/2, \Gamma_h} \leq C \| w \|_{1/2, \Gamma_h},
\end{equation*}
where the constant $C$ does not depend on the mesh size $h$.
\end{lemma}
\begin{proof}
The projection operator $\pi_h$ is linear and maps $H^0(\Gamma_h)$ to $H^0(\Gamma_h)$ and $H^1(\Gamma_h)$ to $H^1(\Gamma_h)$. Furthermore, for $v_0 \in H^0(\Gamma_h)$ and $v_1 \in H^1(\Gamma_h)$, we have that $\|\pi_h(v_0)\|_{0,\Gamma_h} \le C_0 \|v_0\|_{0,\Gamma_h}$ and $\|\pi_h(v_1)\|_{1,\Gamma_h} \le C_1 \|v_1\|_{1,\Gamma_h}$; see \cite[Lemma 1.131]{Ern_04_BOOK}.
Using the operator interpolation theory and in particular \cite[Proposition 14.1.5]{brenner2007mathematical}, we have that $\pi_h$ maps $H^{1/2}(\Gamma_h)$ to $H^{1/2}(\Gamma_h)$ and that its norm satisfies
$$
\|\pi_h\|_{H^{1/2} \rightarrow H^{1/2}} \le \|\pi_h\|_{H^0 \rightarrow H^0}^{\frac12} \|\pi_h\|_{H^1 \rightarrow H^1}^{\frac12} \le \sqrt{C_0 C_1}.
$$
Hence, for $w \in H^{1/2}(\Gamma_h)$ we have 
$$
\|\pi_h w \|_{1/2, \Gamma_h} \leq \|\pi_h\|_{H^{1/2} \rightarrow H^{1/2}} \| w \|_{1/2, \Gamma_h} \leq \sqrt{C_0 C_1} \| w \|_{1/2, \Gamma_h}.
$$
\end{proof}

%%%%%%%%%%%%%%%%%%%%%%%%%%%%%%%%%

\section{Proofs of well-posedness results}\label{wpp}

\subsection{Proof of Theorem \ref{theorem: Dirichlet weak coercivity}}\label{dwc}

The bound \eqref{eqn: Dirichlet boundedness} for $u,v \in {V^k_h}$ is derived simply by seeing that 
\begin{equation*}
\begin{aligned}
	B_{h, D}^\scal(u,v) &= \WTD(u, v-\mathcal R_h v_\star) + \scalh \big< T^k_h(u) \big|_{\etab(\xib)}, v \big>_{\Gamma_h}\\
	&\leq C_{{\overline p, \mathcal R_h}} | u |_{1, \Omega_h} \|v\|_{1, \Omega_h} + \scalh \| T^k_h(u) \big|_{\etab(\xib)} \|_{0,\Gamma_h} |v|_{1, \Omega_h},
\end{aligned}
\end{equation*}   
where use is made of the trace inequality and Lemma \ref{lm: Taylor discrete bound}.

% COERCIVITY
To show that \eqref{eqn: Dirichlet coercivity condition} holds under the hypotheses of the theorem, note that, after applying Lemma \ref{lm: Taylor discrete bound}, we have
\begin{equation*}
\begin{aligned}
	\big< T^k_h(u) \big|_{\etab(\xib)}, u \big>_{\Gamma_h} &= 
	\int_{\Gamma_h} u^2 ds + \big<\ T^k_h(u) \big|_{\etab(\xib)} - u, u \big>_{\Gamma_h} 
	\\&
	\geq
	\|u \|^2_{0,\Gamma_h} -  \big\|  T^{1,k}_h(u) \big|_{\etab(\xib)}\big\|_{0,\Gamma_h}  \|u \|_{0,\Gamma_h} \\
	&
	\geq
	\|u \|^2_{0,\Gamma_h} - C  \Big(\sum^k_{|\alphab| = 1} h^{\frac12-|\alphab|} \deltah^{|\alphab|} \Big)|u|_{1,\Omega_h} \|u \|_{0,\Gamma_h}.
\end{aligned}
\end{equation*}
Second, after applying the Cauchy-Schwarz inequality, Young's inequality, and the inverse inequality, we have that
\begin{equation*}
\begin{aligned}
	&{\WTD}(u, u - \mathcal R_h u_\star) = \WTD(u,u) - \WTD(u, \mathcal R_h u_\star) \\
		&\qquad
		\geq {\urho} |u|^2_{1, \Omega_h} - {\orho} \|\mathcal R_h \||u|_{1, \Omega_h} \|u\|_{1/2, \Gamma_h} 
		\geq \frac{3}{4} {\urho} |u|^2_{1, \Omega_h} - \frac{C_{inv} {\orho} \| \mathcal R_h \|}{{\urho} h} \| u \|^2_{0, \Gamma_h}.
\end{aligned}
\end{equation*}
We then have that
\begin{equation*}
\begin{aligned}
	B^\scal_{h,D}(u,u) &\geq  \frac{3}{4}{\urho} |u|^2_{1, \Omega_h} + \Big(\scalh  - \frac{C_{inv} {\orho} \| \mathcal R_h \|}{{\urho} h}\Big) \| u \|^2_{0, \Gamma_h} \\
	&\qquad- C \scalh \Big (\sum^k_{|\alphab| = 1} h^{\frac12-|\alphab|} \deltah^{|\alphab|} \Big)|u|_{1,\Omega_h} \|u \|_{0,\Gamma_h}.
\end{aligned}
\end{equation*}
If we assume that $\scalh = \dfrac{C_\scal}{h}$ with $C_\scal \ge \frac{C_{inv} \orho \| \mathcal R_h \|}{\urho}$, we have by the Friedrich's inequality
$$
\frac{3}{4}\urho |u|^2_{1, \Omega_h} + \Big(\scalh  - \frac{C_{inv} \orho \| \mathcal R_h \|}{\urho h}\Big) \| u \|^2_{0, \Gamma_h} \ge c\| u\|_{1, \Omega_h}^2;
$$
see \cite[Lemma B.63]{Ern_04_BOOK}. Hence 
\begin{equation*}
	B^\scal_{h, D}(u,u) \geq  c\| u\|_{1, \Omega_h}^2 -  C C_{\scal} \Big (\sum^k_{|\alphab| = 1} h^{-\frac12-|\alphab|} \deltah^{|\alphab|} \Big)|u|_{1,\Omega_h} \|u \|_{0,\Gamma_h}.
\end{equation*}
The coercivity condition {\eqref{eqn: Dirichlet coercivity condition}} is achieved given $\delta_h \sim o(h^{\frac32})$ and $h$ small enough.

%
% STABILITY BOUND
We now prove the stability bound \eqref{eqn: Dirichlet data control}. Let $u_h$ denote the solution to \eqref{eqn: PEF Dirichlet problem} and $\pi_h: L^2(\Gamma_h) \rightarrow W^k_h$ denote the $L^2(\Gamma_h)$ projection operator onto $W^k_h$. Equation \eqref{eqn: PEF Dirichlet problem}$_2$ is equivalent to
\begin{equation*}
	u_h(\xib) = \pi_h \big[ g_D \circ \etab(\xib)\big]  - \pi_h \Big[  T^{1,k}_h(u) \big|_{\etab(\xib)}\Big].
\end{equation*}
Because of this we have that 
\begin{equation} \label{eqn: inequality dirichlet 1}
\begin{aligned}
	\| u_h \|_{1/2, \Gamma_h} &\le \|\pi_h  [ g_D \circ \etab(\xib) ] \|_{1/2, \Gamma_h}  + \Big\| \pi_h \Big[  T^{1,k}_h(u_h) \big|_{\etab(\xib)}\Big] \Big\|_{1/2, \Gamma_h}\\
                              &  \le\|\pi_h [ g_D \circ \etab(\xib) ] \ \|_{1/2, \Gamma_h}  +C  h^{-\frac12} \Big \| \pi_h \Big[  T^{1,k}_h(u_h) \Big|_{\etab(\xib)} \Big] \Big\|_{0, \Gamma_h}\\
                                  &\le C \Big( \Big\| g_D \circ \etab(\xib) \Big\|_{1/2, \Gamma_h}  +  h^{-\frac12}\Big\|   T^{1,k}_h(u_h) \Big|_{\etab(\xib)}\Big\|_{0, \Gamma_h} \Big)\\
                                  &\le C \Big( \Big\| g_D \circ \etab(\xib) \Big\|_{1/2, \Gamma_h}  +      \big(\sum^k_{|\alphab| = 1} h^{-|\alphab| - \frac12} \deltah^{|\alphab|} \big) |u_h|_{1,\Omega_h} \Big)
\end{aligned}
\end{equation}
{after utilizing the inequality $\|\pi_h w\|_{1/2, \Gamma_h} \leq C \| w \|_{1/2, \Gamma_h}$; see Lemma \ref{continuity of projection}, Lemma \ref{lm: Taylor discrete bound}, and the inverse inequality.}

Consider now that
\begin{equation}\label{rhoP}
\begin{aligned}
 &\urho |u_h|^2_{1, \Omega_h} - \orho \|\mathcal R_h \||u_h|_{1, \Omega_h} \|u_h\|_{1/2, \Gamma_h}
\\&\qquad
\leq \WTD(u_h, u_h -  \mathcal R_h (u_h)_\star) = \langle \widetilde f,\, u_h - \mathcal R_h (u_h)_\star\rangle_{\Omega_h}.
\end{aligned}
\end{equation}
From this we obtain
\begin{equation}  \label{eqn: inequality dirichlet 2}
    |u_h|_{1, \Omega_h}^2  \leq C (\| \widetilde f \|_{-1, \Omega_h} + \| u_h\|_{1/2, \Gamma_h} ) \|u_h\|_{1, \Omega_h}.
\end{equation}
Adding $\| u_h\|_{1/2, \Gamma_h}^2$ to both sides of the inequality, using the trace inequality, and redefining the constant $C$, we obtain
$$
|u_h|_{1, \Omega_h}^2 + \| u_h\|_{1/2, \Gamma_h}^2  \leq C (\| \widetilde f \|_{-1, \Omega_h} + \| u_h\|_{1/2, \Gamma_h} ) \|u_h\|_{1, \Omega_h}
$$
We have that $\|u_h\|_{1, \Omega_h}^2 \le C (|u_h|_{1, \Omega_h}^2 + \| u_h\|_{1/2, \Gamma_h}^2)$; see Thanks to \cite[Lemma B.63]{Ern_04_BOOK}. Therefore,
$$
\|u_h\|_{1, \Omega_h} \le C (\| \widetilde f \|_{-1, \Omega_h} + \| u_h\|_{1/2, \Gamma_h} ).
$$
Subsequently substituting $\|u_h\|_{1/2, \Gamma_h}$ for its upper bound in \eqref{eqn: inequality dirichlet 1} yields
\begin{equation*}
	\|u_h\|_{1, \Omega_h} - C  \Big(\sum^k_{|\alphab| = 1} h^{-|\alphab| -\frac12} \deltah^{|\alphab|} \Big)
						\|u_h \|_{1, \Omega_h}  \leq C\left(\| \widetilde f \|_{-1, \Omega_h} + \| g_D \circ \etab(\xib) \|_{1/2, \Gamma_h}\right).
\end{equation*}
It then follows that \eqref{eqn: Dirichlet data control} is satisfied if  $\delta_h \sim o(h^{\frac32})$ and $h$ is small enough.

%%%%%%%%%%%%%%%%%%%%%%%%%
\subsection{Proof of Theorem \ref{theorem: Neumann well--posednessn}}\label{nwc}

Recall that $B_{h,N}(u,v) = \WTN(u,v) + \boldsymbol{\tau}_{N}(u,v)$, where
$$\boldsymbol{\tau}_N(u, v) := \big< \widetilde p\circ \etab(\xib)\;  \mathbf T_h^{k-1} ( \nabla u) \big|_{\etab(\xib)} \cdot \mathbf n - \widetilde p(\xib)\nabla u \cdot \nn_h, v \big>_{\Gamma_h}.$$ 
We have
\begin{equation*}
\begin{aligned}
	|\boldsymbol{\tau}_N(u, v)|
        &\le \left|\left< \left(\widetilde p\circ \etab(\xib) - \widetilde p(\xib)\right)\nabla u \cdot \nn_h, v \right>_{\Gamma_h}\right| + \Big|\Big< \widetilde p\circ \etab(\xib)\nabla u \cdot (\nn - \nn_h), v \Big>_{\Gamma_h}\Big|\\
       & \qquad+  \Big| \Big< \widetilde p\circ \etab(\xib)  \Big (  \mathbf T_h^{k-1} ( \nabla u) \Big|_{\etab(\xib)} - \nabla u \Big)\cdot \mathbf n , v  \big>_{\Gamma_h}\Big|. \\
\end{aligned}
\end{equation*}
Because $|\widetilde p(\etab(\xib)) - \widetilde p(\xib) | \le C \deltah \|\widetilde p\|_{C^1(\overline{\Omega}_h)}$ and $|\widetilde p(\etab(\xib))(\nn - \nn_h)| \le Ch \|\widetilde p\|_{C^0(\overline{\Omega}_h)}$, we have that
\begin{equation}\label{eqn: tau bound}
\begin{aligned}
	|\boldsymbol{\tau}_N(u, v)|
        &\le C_p (\deltah + h)\|\nabla u\|_{0,\Gamma_h} \|v\|_{0, \Gamma_h} \\\
       & + C_p \Big\|  \mathbf T_h^{k-1} ( \nabla u) \Big|_{\etab(\xib)} - \nabla u  \Big\|_{0,\Gamma_h} \|v\|_{0, \Gamma_h}.
\end{aligned}
\end{equation}
Lemma \ref{lm: u approx} implies that $\|\nabla u\|_{0,\Gamma_h} \le C h^{-\frac12} \|\nabla u\|_{0,\Omega_h} + h^{\frac{3}{2}} \trino \nabla u\trino _{2,\Omega_h}$. Combining this with Lemma \ref{lm: Taylor approximation 2} yields the upper bound
$$
\begin{aligned}
|\boldsymbol{\tau}_N(u,v)|
	&\le C\Big[ ( \delta_h h^{-\frac12} + h^\frac12  ) \|\nabla u\|_{0, \Omega_h} + 
		 ( \delta_h h^\frac32 + h^\frac52  ) \trino \nabla u \trino_{2, \Omega_h} \\
	&\qquad + \delta_h^\frac12 \trino \nabla u \trino_{1, \Omega_h} + \delta_h^{k-\frac12} \trino \nabla u \trino_{k, \Omega_h}\Big] \|v\|_{0, \Gamma_h},
\end{aligned}
$$
{where \eqref{eqn: Neumann continuity bound in H1} is derived considering that $B_{h,N}(u,v) = {\WTN}(u,v) + \boldsymbol{\tau}_N(u, v)$, ${\WTN}(u,v) \le M \| u \|_{1, \Omega_h} \| v \|_{1, \Omega_h}$, and the assumption that $\delta_h \sim o(h^{\frac32})$.}

If $u \in V^k_h$, the discrete continuity bound \eqref{eqn: Neumann continuity bound} is derived by proceeding as in the above paragraph but now  using Lemma \ref{lm: Taylor discrete bound} in \eqref{eqn: tau bound} to see that
\begin{equation}\label{eqn: Discrete term bound}
\begin{aligned}
\Big\|   \mathbf T_h^{k-1} ( \nabla u) \Big|_{\etab(\xib)} - \nabla u  \Big\|_{0,\Gamma_h}
&= \Big\|    \mathbf T_h^{1,k-1} ( \nabla u) \Big|_{\etab(\xib)} \Big\|_{0,\Gamma_h}\\ 
&\leq C_P\Big( \sum^k_{|\alphab|=1}h^{-\frac12 - |\alphab|} \delta_h^{|\alphab|} \Big) 
\leq C \|u\|_{0, \Omega_h},
\end{aligned}
\end{equation}
given that $\delta_h \sim o(h^{\frac32})$.

We now show that \eqref{eqn: Neumann coercivity} holds. We have that 
\begin{equation} \label{eqn: lower bound 2}
	{\WTN}(u, u) \geq C_{\urho, \ukappa} \|u\|^2_{1, \Omega_h} \quad \forall u \in V^k_h,
\end{equation}
where $C_{\urho, \ukappa} = \min(\urho,\ukappa)$. Using \eqref{eqn: lower bound 2}, the discrete bound on $|\boldsymbol{\tau}_N(u, v)|$ (i.e., \eqref{eqn: tau bound} with \eqref{eqn: Discrete term bound}), and the discrete trace inequality $\|v\|_{0,\Gamma_h} \leq C h^{-\frac12}\|v\|_{0, \Omega_h}$ for all $ v \in V^k_h$, yields
\begin{equation*}
	B_{h, N}(u,u) \geq \Big[C_{ \urho, \ukappa} - C_\orho \Big(
	 (\deltah h^{-\frac12} + h^\frac12)  + \sum^k_{|\alphab| = 1} h^{-\frac12-|\alphab|} \deltah^{|\alphab|} \Big) \Big] \|u\|^2_{1, \Omega_h}
\end{equation*}
for all $u \in V^h_k$. This implies \eqref{eqn: Neumann coercivity} if $\delta_h \sim o(h^\frac32)$.

Finally, using \eqref{eqn: Neumann coercivity}, one obtains
\begin{equation*}
\begin{aligned}
	\| u_h \|_{1, \Omega_h}^2 &\leq \gamma_N^{-1}B_{h, N}(u_h, u_h) 
	= {\gamma_N^{-1}}\left(F_h(u_h) + \left<g_N\circ \etab(\xib), u_h\right>_{\Gamma_h}\right)\\
	&\leq C {\gamma_N^{-1}}\left( \| \widetilde f \|_{-1, \Omega_h} + \|g_N\circ\etab(\xib) \|_{-1/2, \Gamma_h} \right) \|u_h\|_{1, \Omega_h}
\end{aligned}
\end{equation*}
and hence \eqref{eqn: Neumann data control} is satisfied.

%%%%%%%%%%%%%%%%%

\section{Proofs of error estimates}\label{sec:aeee}

In this section, we provide the proofs of the error estimates given in Theorems \ref{theorem: Dirichlet PE-FEM H1}, \ref{theorem: Dirichlet PE-FEM L2}, and \ref{theorem: Neumann PE-FEM H1}.

\subsection{Proof of Theorem \ref{theorem: Dirichlet PE-FEM H1}}\label{sec:aeeed}

We assume that \eqref{eqn: weak Dirichlet} has a solution $u\in H^{k+1}(\Omega)$ with an extension $\tu \in H^{k+1}(\mathbb R^\DDD)$. We let $\widehat f = L_D \widetilde u$, where $L_D$ denotes the strong operator in \eqref{contprobd}. The error analysis follows the familiar strategy of breaking up $u_h$ into a sum of an interpolant $u_I\in V^k_h$ of  $\widetilde u$ and a discrete error term $w_h = u_h - u_I$. We start with a preliminary result that is central to showing optimal convergence in the $H^1(\Omega_h)$-norm.

\begin{lemma} \label{lm: stability for error bound}
Let $u_h \in V^k_h$ denote the solution of the Dirichlet PE-FEM problem \eqref{eqn: PEF Dirichlet problem} and assume that $g_D \circ \etab(\xib) \in L^2(\Gamma_h)$. Then, we have the stability bound
\begin{equation}
	\| u_h\|_{1, \Omega_h} \leq C\big( \|\widetilde f\|_{-1, \Omega_h} + h^{-\frac12} \|g_D\circ \etab(\xib)\|_{0, \Gamma_h} \big).
\end{equation}
\end{lemma}

\begin{proof}
From \eqref{rhoP} we deduce that
\begin{equation*}
	\urho |u_h|^2_{1, \Omega_h} - C\orho |u_h|_{1, \Omega_h} \|u_h\|_{1/2, \Gamma_h} \leq C \|\widetilde f\|_{-1, \Omega_h} \|u_h\|_{1, \Omega_h}.
\end{equation*}
Applying the inverse inequality and Cauchy's inequality gives us
\begin{equation} \label{eqn: bound a}
	\frac{\urho}{2} |u_h|^2_{1, \Omega_h} - \frac{C h^{-1}}{2\urho} \|u_h\|_{0, \Gamma_h}^2 \leq C \|\widetilde f\|_{-1, \Omega_h} \|u_h\|_{1, \Omega_h}.
\end{equation}
Because we have that 
\begin{equation*}
	 u_h |_{\Gamma_h} = \pi_h \Big[ g_D\circ\etab(\xib) -  T^{1, k}_h (u_h)\big|_{\etab(\xib)} \Big],
\end{equation*}
where $\pi_h(\cdot): L^2(\Gamma_h)\rightarrow W^k_h$ is the $L^2(\Gamma_h)$ projection operator into $W^k_h$, it follows from Lemma \ref{lm: Taylor discrete bound} and the continuity of the projection operator in the $L^2(\Gamma_h)$ topology that
\begin{equation} \label{eqn: bound 2}
	\|u_h\|_{0, \Gamma_h} \leq C\sum_{|\alphab|=1} h^{ {-|\alphab| + \frac12}} \delta_h^{|\alphab|} |u_h|_{1, \Omega_h} + C \|g_D \circ \etab(\xib) \|_{0, \Gamma_h}.
\end{equation}
Inserting this inequality into \eqref{eqn: bound a} yields
\begin{equation*}
	\Big(\frac{\urho^2}{4} - \big(C\sum_{|\alphab|=1} h^{{-|\alphab|}} \deltah^{|\alphab|} \big)^2\Big) |u_h|_{1, \Omega_h}^2 \leq C \| \widetilde f\|_{-1, \Omega_h} \|u_h\|_{1, \Omega_h} + \frac{C h^{-1}}{2\urho} \|g_D \circ \etab(\xib) \|_{0, \Gamma_h}^2.
\end{equation*}
Adding \eqref{eqn: bound 2} to this inequality and subsequently using Friedrich's inequality (i.e. $C\|u\|_{1, \Omega_h}^2 \leq \|u\|_{0, \Gamma_h}^2 + |u|_{1, \Omega_h}^2 $) and the Cauchy inequality yields
\begin{equation*}
	\|u_h\|^2_{1, \Omega_h} \leq C\big( \|\widetilde f \|_{-1, \Omega_h}^2 + h^{-1} \| g_D\circ\etab(\xib)\|_{0, \Gamma_h}^2 \big)
\end{equation*}
after applying the assumption that $\delta_h \sim o(h)$.
\end{proof} 

We proceed with the main result.

%%%%%%%%%%%%%%%%%%%%
\vskip5pt
\paragraph{\textbf{Proof of Theorem \ref{theorem: Dirichlet PE-FEM H1}}}
Application of the triangle inequality yields an error bound
\begin{equation}
	\|\widetilde u - u_h\|_{1, \Omega_h} \leq \|\widetilde u - u_I\|_{1, \Omega_h} + \|w_h\|_{1, \Omega_h}
\end{equation}
in terms of the \emph{interpolation error} and the \emph{discrete error}, where $w_h=u_h-u_I$. Standard interpolation results imply optimal convergence of the first term. Thus, to complete the proof it remains to show that the discrete error is also optimal. By linearity, $w_h \in V^k_h$ satisfies the equation
\begin{equation*}
\begin{aligned}
	B^{\scal}_{h,D}(w_h, v) &= \left< \widetilde f - \left[-\nabla \cdot \left(\widetilde p(x) \nabla u_I\right)\right], v - \mathcal R_h v_\star \right>_{\Omega_h} \\
			&\qquad+ \theta_h \left< g_D \circ \etab(\xib) - \left.T^k_h (u_I) \right|_{\etab(\xib)}, v  \right>_{\Gamma_h} \quad \forall v \in V^k_h
\end{aligned}
\end{equation*}
so that Lemma \ref{lm: stability for error bound} implies the stability bound
\begin{equation}\label{whest}
\begin{aligned}
	\|w_h\|_{1, \Omega_h} &\leq C \Big\{ \| \widetilde f  - \left[-\nabla \cdot  \left(\widetilde p\nabla u_I\right)\right] \|_{-1, \Omega_h} 
	\\&\qquad+ h^{-\frac12}\|g_D \circ \etab(\xib) - \left.T^k_h (u_I) \right|_{\etab(\xib)}\|_{0, \Gamma_h} \Big\}.
\end{aligned}
\end{equation}
Recalling that $\widehat f = L_D\widetilde u$, the first term on the right-hand side of \eqref{whest} satisfies the bound 
\begin{equation} \label{eqn: bound 1}
	\| \widetilde f - \left[-\nabla \cdot \left(\widetilde p\nabla u_I\right)\right] \|_{-1, \Omega_h} \leq \| \widetilde f - \widehat f \|_{-1, \Omega_h} + 
		\|\nabla \cdot \left[\widetilde p\nabla(\widetilde u - u_I)\right]\|_{-1,\Omega_h}.
\end{equation}
From Lemma \ref{lem:ext error}, we have that
\begin{equation*}
	\| \widetilde f - \widehat f\|_{-1, \Omega_h} = \sup_{\substack{\chi \in H^1_0(\Omega_h)\\\|\chi\|_{1, \Omega_h} = 1}} \left<  \widetilde f - \widehat f, \chi\right>_{\Omega_h}\leq 
	\begin{cases}
	C_{k, d} \deltah^{ k - \frac12}\|f\|_{k-1, \Omega} &\quad \textrm{if $d = 2$} \\
	C_{k,d} \deltah^{k-1}\|f\|_{k-1, \Omega} &\quad \textrm{if $d = 3$},
	\end{cases}
\end{equation*}
whereas standard interpolation theory implies that
\begin{equation*}
\begin{aligned}
	\|\nabla \cdot \left[\widetilde p\nabla(\widetilde u - u_I)\right]\|_{-1,\Omega_h} &=   \sup_{\substack{\chi \in H^1_0(\Omega_h)\\\|\chi\|_{1, \Omega_h} = 1}} -  \left< \widetilde p \nabla(\tu - u_I), \nabla \chi \right>_{\Omega_h}\\
	&\leq C\|\widetilde u - u_I\|_{1, \Omega_h} \leq Ch^{k}\| u \|_{k+1, \Omega}
\end{aligned}
\end{equation*}
after applying the extension bound, i.e., $\|\widetilde u\|_{k+1, \Omega_h} \leq C\| u\|_{k+1, \Omega}$.
As a result, 
\begin{equation} \label{eqn: first term bound}
	 \| \widetilde f - \left[-\nabla\cdot \left( \widetilde p \nabla u_I\right)\right]\|_{-1, \Omega_h}  \leq Ch^k \left(\| u \|_{k+1, \Omega} + \|f\|_{k-1, \Omega}  \right),
\end{equation}
after recalling that $\deltah \sim o(h^\frac32)$ if $d=2$ and $\deltah \sim O(h^2)$, if $d = 3$.
 
We now take $\deltah \sim O(h^\frac32)$ for $d=2,3$ to simplify the remainder of the proof. To estimate the second term on the righthand side of \eqref{whest} we first see that
\begin{equation*}
\begin{aligned}
	&\|g_D\circ \etab(\xib) - \left.T^k_h(u_I)\right|_{\etab(\xib)} \|_{0, \Gamma_h} \\
	&\qquad\leq 
		\|g_D\circ\etab(\xib) - \left.T^k_h(\widetilde u)\right|_{\etab(\xib)} \|_{0, \Gamma_h}
		+ \|\left.T^k_h(\widetilde u - u_I)\right|_{\etab(\xib)} \|_{0, \Gamma_h}.
\end{aligned}
\end{equation*}
By Lemma \ref{lm: Taylor approximation}, we have that 
\begin{equation*}
\begin{aligned}
	\|g_D\circ \etab(\xib) - \left.T^k_h(\widetilde u)\right|_{\etab(\xib)} \|_{0, \Gamma_h} &=
	\|\widetilde u\circ \etab(\xib) - \left.T^k_h(\widetilde u)\right|_{\etab(\xib)} \|_{0, \Gamma_h}\\
	&\leq C_{\Omega_h}\deltah^{k+\frac12}|\widetilde u|_{k+1, \mathbb R^\DDD} \\
	&\leq C_{\Omega_h}h^{\frac{3k}{2} + \frac34}\|u\|_{k+1, \Omega},
\end{aligned}
\end{equation*}
after applying the assumption $\delta \sim o(h^\frac32)$ in Theorem \ref{theorem: Dirichlet weak coercivity} and the extension bound. Additionally, let $e_I = \widetilde u - u_I$ be the interpolation error defined over $\Omega_h$ , then we have that
\begin{equation*}
	\|\left. T^k_h(e_I)\right|_{\etab(\xib)} \|_{0, \Gamma_h} \leq 
		\|\left. T^k_h( e_I) \right|_{\etab(\xib)} - e_I \|_{0, \Gamma_h} 
		+ \|e_I \|_{0, \Gamma_h}.
\end{equation*}
By Lemma \ref{lm: Taylor approximation 2}, we have that
\begin{equation*}
\begin{aligned}
	\|\left. T^k_h( e_I) \right|_{\etab(\xib)} - e_I \|_{0, \Gamma_h} &\leq 
	C_{\Omega_h} \left( \deltah^\frac12 \|e_I\|_{1, \Omega_h} + \deltah^{k+\frac12} \|e_I\|_{k+1, \Omega_h} \right)\\
	&\leq C_{\Omega_h} \left(\deltah^{\frac12}h^k + \deltah^{k+\frac12} \right) \|u\|_{k+1, \Omega} \\
	&\leq C_{\Omega_h} h^{k+\frac34} \|u\|_{k+1, \Omega},
\end{aligned}
\end{equation*}
where we have again applied the assumption that $\delta \sim o(h^\frac32)$ and the extension bound. Further, using the following trace inequality \cite[Theorem 1.6.6]{brenner2007mathematical} for $v \in H^1(\mathcal D)$ and $\mathcal D \subset \mathbb R^\DDD$ a Lipschitz domain
\begin{equation} \label{eqn: Necas trace inequality}
	\|v \|_{0, \partial \mathcal D} \leq  C \|v\|_{0, \mathcal D}^\frac12\|v\|_{1, \mathcal D}^\frac12,
\end{equation} 
we have that
\begin{equation*}
	\|e_I\|_{0, \Gamma_h} \leq C_{\Omega_h} \|e_I\|^\frac12_{0, \Omega_h} \|e_I\|^\frac12_{1, \Omega_h}
		\leq C_{\Omega_h} h^{k+\frac12}\|u\|_{k+1, \Omega},
\end{equation*}
after applying the standard interpolation error bounds and extension bounds. Thus, from this analysis, we have that
\begin{equation} \label{eqn: second term bound}
	\|g_D\circ\etab(\xib) - \left. T^k_h(u_I) \right|_{\etab(\xib)} \|_{0, \Gamma_h} \leq Ch^{k+\frac12} \|u\|_{k+1, \Omega}.
\end{equation}
Combining \eqref{eqn: first term bound} and \eqref{eqn: second term bound} yields the optimal bound
\begin{equation}\label{eqn: perturbation error}
	\|w_h\|_{1, \Omega_h} \leq Ch^k\left\{ \|u\|_{k+1, \Omega} + \|f \|_{k-1, \Omega} \right\}
\end{equation}
for the discrete error which completes the proof.

%%%%%%%%%%%%%%%
\subsection{Proof of Theorem \ref{theorem: Dirichlet PE-FEM L2}}\label{sec:aeeed2}

Our strategy is to couch the PE-FEM problem into a standard Dirichlet finite element formulation, under the additional assumption that $u\in H^{k+\frac32}(\Omega)$ and $f \in H^{k}(\Omega)$ (and thus, $\widetilde u \in H^{k+\frac32}(\mathbb R^\DDD)$ and $\widetilde f \in H^k(\mathbb R^\DDD)$), and then apply the well-known Aubin-Nitsche duality argument to the latter.

We begin with two technical lemmas.

\begin{lemma} \label{lm: perturbation trace bound}
Assume the hypotheses in Theorem \ref{theorem: Dirichlet weak coercivity}, and in addition assume that $u \in H^{k+\frac32}(\Omega)$ and $\deltah \sim O(h^2)$. Then, the trace of the discrete error $(w_h)_\star \in W^k_h$ satisfies the bound
	\begin{equation}
		\|w_h\|_{0, \Gamma_h} \leq Ch^{{k+1}}\left( \|u\|_{k+1, \Omega} + |\widetilde u|_{k+1, \Gamma_h} + \|f\|_{k-1, \Omega} \right).
	\end{equation}
\end{lemma}

\begin{proof}
First, we begin with a technical result that arises from the assumption that ${\widetilde u} \in H^{k+\frac32}(\mathbb R^\DDD)$ and $\deltah \sim O(h^2)$. Under these assumptions we may modify the analysis on the $\|g_D\circ \etab(\xib) - \left.T^k_h(\widetilde u)\right|_{\etab(\xib)} \|_{0, \Gamma_h}$ term in the proof of Theorem \ref{theorem: Dirichlet PE-FEM H1} so that
\begin{equation} \label{eqn: essential condition error}
	\|g_D\circ \etab(\xib) - \left.T^k_h({u_I})\right|_{\etab(\xib)} \|_{0, \Gamma_h} 
		\leq Ch^{k+1} \left( \|u\|_{k+1, \Omega} + |\widetilde u|_{k+1, \Gamma_h}\right).
\end{equation}
This result is immediate after seeing that
\begin{equation*}
	\|\widetilde u - u_I \|_{0, \Gamma_h} \leq Ch^{k+1} |\widetilde u|_{k+1, \Gamma_h},
\end{equation*}
and taking $\deltah \sim O(h^2)$. 

From \eqref{eqn: PEF Dirichlet problem}, we have that 
\begin{equation*}
	w_h = \pi_h\Big(g_D \circ \etab(\xib) -  T^k_h(u_I)\big|_{\etab(\xib)} -  T^{1,k}_h(w_h)\big|_{\etab(\xib)} \Big),
\end{equation*}
where $\pi_h: L^2(\Gamma_h) \rightarrow W^k_h$ is the $L^2$ projection onto $W^k_h$. Using the continuity of $\pi_h$  in $L^2(\Gamma_h)$, \eqref{eqn: essential condition error}, Lemma \ref{lm: Taylor discrete bound}, and \eqref{eqn: perturbation error}, we have that
\begin{equation*}
\begin{aligned}
	&\|w_h\|_{0, \Gamma_h} \leq C\bigg( \|g_D\circ\etab(\xb) - \left.T^k_h(u_I) \right|_{\etab(\xib)} \|_{0, \Gamma_h} + 
		\left\|\left.T^{1,k}_h(w_h)\right|_{\etab(\xib)} \right\|_{0, \Gamma_h}\bigg) \\
	&\quad\leq C\bigg( h^{k+1} \left(\|u\|_{k+1, \Omega} + |{\widetilde u}|_{k+1, \Gamma_h}\right)+ \left\|\left. T^{1,k}_h(w_h) \right|_{\etab(\xib)} \right\|_{0, \Gamma_h} \bigg) \\
	&\quad\leq C\bigg( h^{k+1} \left(\|u\|_{k+1, \Omega} + |{\widetilde u}|_{k+1, \Gamma_h}\right) + \sum^k_{|\alphab| = 1} h^{-|\alphab| + \frac12} \deltah^{|\alphab|} \|w_h\|_{1, \Omega_h} \bigg) \\
	&\quad\leq C \bigg( h^{k+1} \left(\|u\|_{k+1, \Omega} + |{\widetilde u}|_{k+1, \Gamma_h}\right) + \sum^k_{|\alphab| = 1} h^{k+|\alphab| + \frac12} \left( \|u\|_{k+1, \Omega} + \|f\|_{k-1, \Omega} \right) \bigg) \\
	&\quad\leq C h^{k+1} \left( \|u\|_{k+1, \Omega}+ |{\widetilde u}|_{k+1, \Gamma_h} + \|f\|_{k-1, \Omega} \right),
\end{aligned}
\end{equation*}
thereby proving the result of this lemma. 
\end{proof}

\begin{lemma} \label{lm: perturbation function}
Assume that $u \in H^{k+\frac32}(\Omega)$, $\delta_h \sim O(h^2)$, and that $\phi_h\in V^k_h$ satisfies the conditions
\begin{equation*}
		\left.\phi_h\right|_{\Gamma_h} = w_h, \quad
		\sup_{\xb \in \Omega_h} |\phi_h| = \sup_{\xib\in \Gamma_h} |w_h|, \quad
		\textrm{ and } \quad
		\phi_h = 0 \textrm{ over } \Omega_h^0,
\end{equation*}
where $ \Omega_h^0 := \left \{ \xb \in \KKK_n \, : \, \KKK_n \cap \Gamma_h = \emptyset  \right\}$. Then, 
\begin{equation*}
		\| \phi_h\|_{0, \Omega_h} \leq C h^{k+\frac32}(\|u\|_{k+1, \Omega} + |\widetilde u|_{k+1, \Gamma_h}+ \|f\|_{k-1, \Omega}) .
\end{equation*}
\end{lemma}
\begin{proof}
Let $\Omega_h^b:= \Omega_h \setminus \Omega_h^0$. Because $\phi_h = 0$ on $\Omega_h^0$, the inverse inequality implies
\begin{equation*}
\begin{aligned}
	\| \phi_h \|^2_{0, \Omega_h} &= \sum_{\KKK_n \in \Omega_h^b} \| \phi_h\|_{0, \KKK_n}^2 \leq \sum_{\KKK_n \in \Omega_h^b} C h^{d}\|\phi_h\|^2_{L^\infty(\KKK_n)} \\
		&= \sum_{i} C h^{d} \|w_h\|^2_{L^\infty(\EEE_i)} \leq \sum_{i} C h\|w_h\|^2_{0, \EEE_i} 
		= C h \|w_h\|^2_{0, \Gamma_h} .
\end{aligned}
\end{equation*}
The proof follows by an application of Lemma \ref{lm: perturbation trace bound}.
\end{proof}

We proceed with the main result.

%%%%%%%%%%%%%%%%%%%%
\vskip5pt
\paragraph{\textbf{Proof of Theorem \ref{theorem: Dirichlet PE-FEM L2}}}
We couch the PE-FEM Dirichlet problem into an equivalent standard FEM Dirichlet problem as follows. Then, for the extension $\widetilde u$ we have that {$\widetilde u \in H^1(\Omega_h)$} satisfies
\begin{equation*}
	\WTD(\widetilde u, v) = \left< \widehat f, v\right>_{\Omega_h} \quad \forall v \in H^1_0(\Omega_h).
\end{equation*}
Let $\phi_h \in V^k_h$ denote a function satisfying the hypothesis of Lemma \ref{lm: perturbation function}. Then, it is not difficult to see that the solution $u_h \in V^k_h$ of the Dirichlet PE-FEM problem \eqref{eqn: PEF Dirichlet problem} is also a solution to the following weak problem: seek {$u_h \in V^k_{h,0} + u_I + \phi_h$} such that
\begin{equation*}
	\WTD(u_h, v) = \left< \widetilde f, v \right>_{\Omega_h} \quad \forall v \in V^k_{h,0},
\end{equation*}
where {$u_I \in V^k_h$ denotes the interpolant of $\widetilde u$ over $\Omega_h$}. These results imply that
\begin{equation} \label{eqn: difference relation}
\WTD(\widetilde u - u_h, v) = \left<\widehat f - \widetilde f, v\right>_{\Omega_h} \quad \forall v \in V^k_{h,0}.
\end{equation}
Let  {$u_h^0 = u_h - u_I - \phi_h$}, and consider the following  continuous and discrete dual problems: seek $\psi \in H^1_0(\Omega_h)$ such that
\begin{equation} \label{eqn: continuous dual problem}
	\WTD(\chi, \psi) = {\left( u^0_h, \chi\right)}_{\Omega_h} \quad
		\forall \chi \in H^1_0(\Omega_h)
\end{equation}
and $\psi_h \in V^k_{h,0}$ such that
\begin{equation} \label{eqn: discrete dual problem}
	\WTD(\chi, \psi_h) = {\left(u^0_h, \chi \right)}_{\Omega_h} \quad \forall \chi \in V^k_{h,0}.
\end{equation}
Under the assumptions made on the domain in the statement of the theorem, we have that \eqref{eqn: continuous dual problem} satisfies the following continuity bound on polytopial domains \cite{dauge:1988,grisvard2011elliptic} 
\begin{subequations}
\begin{equation} \label{eqn: continuous dual stability}
	\|\psi\|_{1 + s, \Omega_h} \leq C{\|u_h^0\|}_{0, \Omega_h},
\end{equation}
where {$s \in (\frac12, 1]$} is a constant dependent on the largest interior angle of $\partial \Omega_h$. If all interior angles are bounded above by $\pi$ we have that {$s = 1$}.
We also have that the solution of \eqref{eqn: discrete dual problem} satsifies
\begin{equation} \label{eqn: discrete dual stability}
	\|\psi_h\|_{1, \Omega_h} \leq C {\| u_h^0\|}_{-1, \Omega_h}.
\end{equation}
Combining these results with standard finite element error bounds obtained from C\'ea's lemma allows us to conclude that
\begin{equation} \label{eqn: dual error}
	\|\psi - \psi_h\|_{1, \Omega_h} \leq C h^s|\psi|_{1+s, \Omega_h} \leq Ch^s{\| u_h^0\|}_{0, \Omega_h}.
\end{equation}
\end{subequations}
Applying Lemmas \ref{lem:ext error} and \ref{lm: perturbation function} together with standard interpolation results, \eqref{eqn: difference relation}, \eqref{eqn: continuous dual stability}, and \eqref{eqn: discrete dual stability} yields 
\begin{equation} \label{eqn: BOUND}
\begin{aligned}
	\|u^0_h&\|^2_{0, \Omega_h} = D_h(u^0_h, \psi) = D_h(u_h, \psi) - D_h(u_I + \phi_h, \psi) \\
		&= D_h(u_h - \widetilde u, \psi) - D_h(u_I - \widetilde u + \phi_h, \psi) \\
		&= D_h(u_h - \widetilde u, \psi - \psi_h) + D_h(u_h - \widetilde u, \psi_h) - D_h(u_I - \widetilde u + \phi_h, \psi) \\
		&= D_h(u_h - \widetilde u, \psi - \psi_h) + \left< \widetilde f - \widehat f, \psi_h \right>_{\Omega_h} -
			\left< u_I - \widetilde u + \phi_h, u^0_h \right>_{\Omega_h} \\
		&\qquad - \left<u_I - \widetilde u + \phi_h, \widetilde p \nabla \psi \cdot \nn_h \right>_{\Gamma_h}\\
		&\leq C\bigg( \|\widetilde u - u_h \|_{1, \Omega_h} \|\psi - \psi_h \|_{1, \Omega_h}
			+ \| \widetilde f - \widehat f\|_{-1, \Omega_h} \|\psi_h\|_{1, \Omega_h} \\
		&\qquad + 
			\left( \|\widetilde u - u_I\|_{0, \Omega_h} + \|\widetilde u - u_I\|_{0, \Gamma_h} + \|\phi_h\|_{0, \Omega_h} + \|\phi_h\|_{0, \Gamma_h} \right) \|u^0_h\|_{0, \Omega_h} \bigg) \\
		&\leq Ch^{k+s}\left( \|u\|_{k+1, \Omega_h} {+ |\widetilde u|_{k+1, \Gamma_h}}  + \|f\|_{k, \Omega_h} \right) \|u_h^0\|_{0, \Omega_h},
\end{aligned}
\end{equation}
after seeing that 
\begin{equation*}
\begin{aligned}
	D_h(u_I - \widetilde u + \phi_h, \psi) &:= \int_{\Omega_j} {\widetilde p\, }\nabla(u_I - \widetilde u + \phi_h)\cdot \nabla \psi d\xb 
	\\&= \int_{\Omega_h} (u_I - \widetilde u + \phi_h )\, L_D\psi d\xb + \int_{\Gamma_h} (u_I - \widetilde u + \phi_h)\, \widetilde p\, \nabla \psi\cdot \nn_h ds
\end{aligned}
\end{equation*}
from Green's identity, recalling that $L_D \psi = u^0_h$, and subsequently having that
\begin{equation*}
\begin{aligned}
	\left< u_I - \widetilde u + \phi_h, \widetilde p \nabla \psi\cdot \nn_h \right>_{\Gamma_h} 
		&\leq \orho\left(\|u_I - \widetilde u\|_{1/2-s, \Gamma_h} + \|\phi_h\|_{1/2 - s, \Gamma_h}\right)
			\|\nabla \psi\|_{s-1/2, \Gamma_h} \\
		&\leq \orho\left(\|u_I - \widetilde u\|_{0, \Gamma_h} + \|\phi_h\|_{0, \Gamma_h}\right)
			\|\psi\|_{1+s, \Omega_h},
\end{aligned}
\end{equation*}
Here, the trace is well--defined since we have taken $s > \frac12$ (see Remark \ref{remark: dual regularity remark}).
To complete the proof we use the above result in conjunction with the triangle inequality to obtain
$$
	\| \widetilde u - u_h \|_{0, \Omega_h} \leq \|u^0_h\|_{0, \Omega_h} + \|\widetilde u - u_I\|_{0, \Omega_h} + \|\phi_h\|_{0, \Omega_h}
$$
from which the result of the theorem follows after applying \eqref{eqn: BOUND}, standard interpolation bounds, and Lemma \ref{lm: perturbation function}.

\begin{remark} \label{remark: dual regularity remark}
Because we have assumed that $\Gamma$ is $C^{k+1}$ smooth, there exists an $h_0$ such that for all $h < h_0$ we have that the largest interior angle of $\Gamma_h$ is bounded above by $\frac{3\pi}{2}$. This implies that $\psi\in H^{1+s}(\Omega_h)$ with $s > \frac12$. Additionally, because the largest interior angle of $\Gamma_h$ approaches $\pi$ as $h \rightarrow 0$, we have that $s \rightarrow 1$ as $h \rightarrow 0$. Therefore the $L^2$ estimate presented above for nonconvex polytopial domains is asymptotically optimal.
\end{remark}

%%%%%%%%%%%%%%%
\subsection{Proof of Theorem \ref{theorem: Neumann PE-FEM H1}}\label{sec:aeeen}
%%%%%%%%%%%%%%%%%%%%

We assume that \eqref{eqn: weak Neumann} has a solution $u\in H^{k+1}(\Omega)$ with an extension $\tu \in H^{k+1}(\mathbb R^\DDD)$. We denote 
$\widehat f = L_N \widetilde u$, where $L_N$ is the strong operator in \eqref{contprobn}. It follows that $\widetilde u$ satisfies 
\begin{equation} \label{Neumann continuos problem}
N_{h}(\widetilde u, v_h) - \left< \widetilde p \, \nabla \widetilde u  \cdot \mathbf n_h, v_h \right>_{\Gamma_h} =\left<\widehat f,v_h\right>_{\Omega_h}, \forall v_h \in V_h^k(\Omega_h).
\end{equation}

From the derivation of Strang's lemma \cite[Theorem 10.1.1]{brenner2007mathematical}, we have that
\begin{equation} \label{eqn: base Strang inequality}
\begin{aligned}
	\|\widetilde u - u_h\|_{1, \Omega_h} &\leq  C_1\inf_{w \in V^k_h}\Big[ \|\widetilde u - w\|_{1, \Omega_h} \\
			&\quad + \sup_{\substack{v \in V^k_h \\ \|v\|_{1, \Omega_h} = 1}} 
			\Big(\left|B_{h,N}(\widetilde u - w, v)\right| + \left| B_{h,N}(\widetilde u, v) - F_{h,N}(v) \right|\Big)\Big].
\end{aligned}
\end{equation}
Setting $w = u_I$, where $u_I \in V^k_h$ is the {Lagrange interpolant} 
of $\widetilde u$, we have, by \eqref{eqn: Neumann continuity bound in H1} and standard interpolation bounds and recalling that we have now required that $\delta_h \leq Ch^2$, that
\begin{equation} \label{eqn: Neumann error bound 1}
 \sup_{\substack{v \in V^k_h \\ \|v\|_{1, \Omega_h} = 1}} 
			\left|B_{h,N}(\widetilde u - w, v)\right| \leq 
			C h^k|u|_{k+1, \Omega_h}.
\end{equation}
From the definition of $B_{h,N}$ and \eqref{Neumann continuos problem}, we have
$$
	B_{h,N}(\widetilde u, v) =  \left <\widehat f, v \right>_{\Omega_h} + \Big< \widetilde p\circ \etab(\xib)\left. \mathbf T^{k-1}_h(\nabla \widetilde u)\right|_{\etab(\xib)}\cdot \nn, v\Big>_{\Gamma_h}. 
$$
Recalling \eqref{eqn: Neumann representation} we have that
$$
	B_{h,N}(\widetilde u, v) = \left<\widehat f, v \right>_{\Omega_h} + \left< g_N\circ \etab(\xib) - \widetilde p\circ\etab(\xib)\left.\mathbf{R}^{k-1}_h(\nabla \widehat u)\right|_{\etab(\xib)} \cdot \nn, v \right>_{\Gamma_h}.
$$
Thus, we have from Lemmas \ref{lm: Taylor approximation} and \ref{lem:ext error} that
\begin{equation}\label{eqn: Neumann error bound 2}
\begin{aligned}
	 &\sup_{\substack{v \in V^k_h \\ \|v\|_{1, \Omega_h} = 1}}  \left| B_{h,N}(\widetilde u, v) - F_{h,N}(v)\right|\\
	 &\qquad =  \sup_{\substack{v \in V^k_h \\ \|v\|_{1, \Omega_h} = 1}} \Big|\left< \widehat f - \widetilde f, v \right>_{\Omega_h} {-} \left< \widetilde p\circ\etab(\xib) \left. \mathbf R^{k-1}_h(\nabla \widetilde u)\right|_{\etab(\xib)} \cdot \nn, v \right>_{\Gamma_h} \Big| \\
	 &\qquad \leq C\delta_h^{k-1}\|f\|_{k-1, \Omega} + C\delta_h^{k+\frac12}|u|_{k+1, \Omega} 
	  \leq Ch^{k} \left(\|f\|_{k-1, \Omega} + |u|_{k+1, \Omega} \right).
\end{aligned}
\end{equation}
The result of the theorem follows by inserting \eqref{eqn: Neumann error bound 1} and \eqref{eqn: Neumann error bound 2} into \eqref{eqn: base Strang inequality} and subsequently applying the approximation bound \eqref{eqn: best approximation bound}.

\end{document}